%% file: main-ssrn.tex
\documentclass{article}
\usepackage{iclr2027_conference,times}
\iclrfinalcopy
\usepackage[english]{babel}
\usepackage{mathtools,amssymb,amsthm}
\usepackage{graphicx,booktabs,tikz}
\usepackage{placeins}
\usetikzlibrary{arrows.meta,positioning,calc}
\usepackage{hyperref}
\usepackage{url}
\hypersetup{hidelinks,pdflang={en-US},
  pdftitle={Tightness and Error Exponents of SDP with Logarithmically Many Communities},
  pdfauthor={Shengtao Guo, Ethan X. Fang, Junwei Lu}}

\title{Tightness and Error Exponents of SDP with Logarithmically Many Communities}
\author{Shengtao Guo
  \qquad
  Ethan X. Fang\thanks{Department of Biostatistics \& Bioinformatics,
  Duke University, Durham, NC 27710, USA. Email: \texttt{ethan.fang@duke.edu}.}
  \qquad
  Junwei Lu\thanks{Department of Biostatistics, Harvard T.H. Chan School of
  Public Health, Boston, MA 02115, USA. Email: \texttt{junweilu@hsph.harvard.edu}.}}
\renewcommand{\headrulewidth}{0pt}
\makeatletter
\def\@maketitle{\vbox{\hsize\textwidth\centering
  {\LARGE\sc \@title\par}
  \begin{tabular}[t]{c}\bf\rule{\z@}{24pt}\@author\end{tabular}\par
  \vskip 0.3in minus 0.1in}}
\makeatother

\newtheorem{theorem}{Theorem}[section]
\newtheorem{proposition}[theorem]{Proposition}
\newtheorem{lemma}[theorem]{Lemma}
\newtheorem{corollary}[theorem]{Corollary}
\theoremstyle{remark}

\numberwithin{equation}{section}
\newcommand{\E}{\mathbb E}
\newcommand{\Prob}{\mathbb P}
\newcommand{\one}{\mathbf 1}
\newcommand{\cH}{\mathcal H}
\newcommand{\cS}{\mathcal S}
\newcommand{\diag}{\operatorname{diag}}
\newcommand{\Tr}{\operatorname{Tr}}
\newcommand{\range}{\operatorname{range}}

\begin{document}
\maketitle
\lhead{}
\input{main-body}

\subsection*{AI use statement}
The Odin Automatic AI Research Agent was used to aid or polish the
writing and to assist with proving mathematical claims. The manuscript has undergone AI-assisted proof
and consistency checks. The authors reviewed the AI-assisted work and take
responsibility for the final content.

\bibliographystyle{iclr2027_conference}
\bibliography{references}
\clearpage
\appendix
\input{appendix}

\end{document}

%% file: main-body.tex
\begin{abstract}
We study a semidefinite programming (SDP) relaxation for community
recovery when the number of communities grows logarithmically.
In the balanced stochastic block model with \(n=km\) vertices, we consider the regime
\(k/\log m\to\gamma>0\), with edge probabilities
\(\alpha\log m/m\) within communities and \(\beta\log m/m\) across them,
for fixed \(\alpha>\beta>0\).
We derive the sharp asymptotic tightness boundary away from critical
cases. When rare vertices cause tightness to fail while the bulk remains
spectrally stable, the normalized matrix error of every near-optimal solution still
vanishes. We prove matching high-probability exponents for this error and
the normalized optimal objective gain, governed by the same local correction that
determines tightness. Throughout this spectrally stable region, a single SDP solve followed
by explicit rounding and refinement achieves exact community recovery
above the information-theoretic threshold.
This guarantee holds even when the planted community matrix is not
an optimal solution to the SDP.
\end{abstract}

\section{Introduction}

Semidefinite relaxations provide a principled approach to community
recovery by optimizing over matrices that represent pairwise community
membership. When a relaxation is tight, its unique optimal solution recovers the
planted community matrix, and the labels follow immediately. Once
tightness fails, however, the statistical consequences are less clear.
How far does the optimizer move from the planted matrix, and can its
remaining structure still support exact recovery? We study these
questions for the balanced cluster-matrix semidefinite program (SDP)
in the stochastic block model (SBM), a canonical model for graph
clustering \citep{Abbe18}.

The regime of logarithmically many communities brings these questions
into focus. For fixed or sublogarithmically many balanced communities,
this SDP attains the information-theoretic exact-recovery threshold
\citep{HajekWuXuExtensions16,AgarwalBandeiraKoiliarisKolla15}. However,
at the logarithmic scale, local degree fluctuations and collective
spectral noise become comparable, and the sharp behavior of the
relaxation is unresolved by those guarantees. Specifically, we consider
\(n=km\) vertices partitioned into \(k\) communities of size \(m\).
Let \(L=\log m\), and assume \(k/L\to\gamma>0\), with within-
and between-community edge probabilities \(\alpha L/m\) and
\(\beta L/m\), respectively, for fixed \(\alpha>\beta>0\).
In this regime, the number of communities grows logarithmically,
while the expected degree of each vertex is of order \(L^2\).
This setting addresses the SDP tightness question raised by
\citet[Open Problem~9.2]{Bandeira15}.

Our results characterize both the loss of tightness and the information
retained beyond it. Let \(a=\alpha-\beta\) and
\(\sigma=\sqrt{\beta\gamma}\). A vertex's degree margin against a
competing community is the difference between its degrees into its
own community and that competing community. In the spectrally stable
region \(a>2\sigma\), the surrounding noise reduces the effective margin
by a correction \(\tau L\), where
$
  \tau=\big(a-\sqrt{a^2-4\sigma^2}\big)/{2}.
$
Let \(I(x)\) denote the lower-tail rate for a degree margin of
size \(xL\), as defined in Section~\ref{sec:main}. We prove that
the SDP is tight with high probability when \(I(\tau)>1\),
whereas the planted matrix is not optimal with high probability
when \(I(\tau)<1\). In this local failure region, both the normalized matrix error and
the normalized optimal objective gain over the planted matrix
have size \(m^{-I(\tau)+o_{\Prob}(1)}\).
The matrix-error guarantee holds uniformly over all feasible
solutions whose objective value is at most 1 below the optimum.

\begin{figure}[t]
  \centering
  \includegraphics[width=\linewidth]{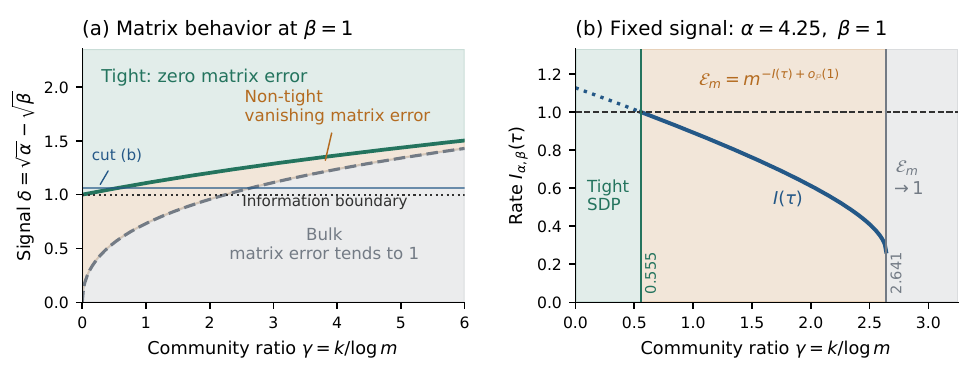}
  \caption{Tightness, matrix error, and the bulk transition.
  (a) Fixed-\(\beta\) regions from Theorems~\ref{thm:main}--\ref{thm:error}
  and Proposition~\ref{prop:bulk-value}; the dotted information boundary
  is prior work \citep{ZhangZhou16}.
  (b) Along the fixed-signal cut in (a), \(I(\tau)\) crosses one at
  \(\gamma\approx0.555\). On the orange interval, matrix error and
  objective gain have exponent \(I(\tau)\). On the green interval,
  the SDP is tight with high probability; the dotted curve continues the
  rate function, not the optimizer's error. The bulk boundary lies at
  \(\gamma\approx2.641\), beyond which \(\tau\) is not real.
  The colored regions describe strict asymptotic regimes.
  Zero-error annotations refer to exact SDP optimizers. Both panels
  are evaluations of theoretical formulas.}
  \label{fig:phase-diagram}
\end{figure}

Even when the planted matrix is not an optimal solution to the SDP,
the SDP output can still retain enough information for exact label
recovery. Throughout the region \(a>2\sigma\), we show that a single
SDP solve followed by explicit row clustering and refinement recovers
all labels with high probability whenever
\(I(0)=(\sqrt{\alpha}-\sqrt{\beta})^2>1\), the
information-theoretic exact-recovery condition. This guarantee
includes the entire information-recoverable part of the local
failure region, where the planted matrix is not optimal.
In the bulk regime \(a<2\sigma\), near-optimal matrices instead
have asymptotically zero normalized overlap with the planted
matrix, although this alone does not rule out label recovery.
Figure~\ref{fig:phase-diagram} summarizes the matrix regimes;
Section~\ref{sec:main} states the precise guarantees and
excluded critical cases.

Our proof identifies a common mechanism behind the tightness
boundary and the local error exponent. First, we use a Schur
complement to account for the surrounding graph, yielding the
correction \(\tau L\). Controlling this correction requires an
operator-norm resolvent estimate on a set of low-margin vertices
selected from the same graph, so independence cannot be assumed.
We then increase selected diagonal entries of the dual certificate
to restore positive semidefiniteness. Controlling the total size
of these adjustments yields upper bounds on both matrix error
and objective gain. For the matching lower bound, we build forests
rooted at many low-margin vertices, balancing their structure
across communities. We use these forests to construct a Gram
matrix that satisfies all SDP constraints and improves on the
planted objective by the required order. The upper and lower
bounds have the same exponent because both depend on how many
vertices have degree margins near \(\tau L\).

For exact label recovery, a weaker matrix-error bound suffices
to obtain an initial labeling with few misclassified vertices.
We then show that refinement reduces the number of errors
whenever the set of misclassified vertices is sufficiently small.
Because this guarantee holds for all such sets simultaneously,
it remains valid when the initial labeling and each refinement
step use the same graph.

\paragraph{Relation to prior work.}
We use the relaxation and dual certificate of
\citet{AgarwalBandeiraKoiliarisKolla15}. Their sufficient condition
covers logarithmically many communities but does not determine
the sharp constants in this regime.
\citet[Section~6, Theorem~17 and Remark~18]{HajekWuXu16}
show that the SDP can fail to be tight above the
information-theoretic exact-recovery threshold when the
logarithmic ratio \(\gamma\) is sufficiently large.
Theorem~\ref{thm:main} determines the sharp tightness boundary
for every fixed \(\gamma>0\), apart from the stated critical
cases, addressing the balanced-model SDP question raised by
\citet[Open Problem~9.2]{Bandeira15}.
Theorem~\ref{thm:error} further quantifies the matrix error
and objective gain in the local failure region.
\citet{ChenXu16} study broader statistical-computational
tradeoffs with a growing number of clusters. Their results
characterize recovery regimes up to constant factors,
whereas our analysis identifies the sharp constants
for the balanced SDP considered here.

Beyond tightness, \citet{FeiChen19} derive exponential bounds on
SDP matrix error and show how to obtain community labels by
clustering the rows of the SDP output. Their analysis also allows
the row constraints used here, but does not determine the sharp
error exponent in our regime. \citet{FeiChenBayes19} obtain sharp
Bayes exponents for the binary symmetric SBM and use a planted
dual certificate to bound errors even when tightness fails.
We build on this approach to obtain matching upper and lower
bounds with exponent \(I(\tau)\) for logarithmically many
communities. To translate matrix-error bounds into label-recovery
guarantees, we combine row clustering with local refinement,
an approach with established precedents
\citep{FeiChen19,GaoMaZhangZhou17}. Here we show that
the non-tight SDP output provides a sufficiently accurate
initialization for exact recovery under the stated conditions,
and justify reusing the same graph for initialization and
refinement. Appendix~\ref{app:comparison} provides detailed
comparisons of the models, feasible sets, and parameter ranges.

\section{Model and main results}
\label{sec:main}
\subsection{Model, SDP, and error measures}

Consider \(n=km\) vertices partitioned into fixed communities
\(C_1,\ldots,C_k\), each of size \(m\). Let \(c(i)\) denote
the community label of vertex \(i\). Under the regime
\begin{equation}
  L=\log m,\qquad k/L\longrightarrow\gamma\in(0,\infty),
  \label{eq:k-scaling}
\end{equation}
the adjacency matrix \(A\) is generated independently for \(i<j\) by
\begin{equation}
  A_{ij}\sim
  \begin{cases}
    \operatorname{Ber}(p),&c(i)=c(j),\\
    \operatorname{Ber}(q),&c(i)\ne c(j),
  \end{cases}
  \quad p=\alpha L/m,\quad q=\beta L/m,\quad \alpha>\beta>0,
  \label{eq:model}
\end{equation}
with \(A_{ji}=A_{ij}\) and \(A_{ii}=0\). Throughout,
\(\alpha,\beta,\gamma\) remain fixed as \(m\to\infty\).

Let \(Z^*\) be the planted community matrix, with entries
\(Z^*_{ij}=\mathbf1_{\{c(i)=c(j)\}}\). Assuming that the community
size \(m\) is known, we consider the semidefinite relaxation
\begin{equation*}
  \begin{aligned}
  V_m=\max_{Z\in\mathbb S^n}\quad\langle A,Z\rangle, \quad \text{subject to}\quad&Z\succeq0,\quad Z\geq0,\quad
        \diag Z=\one,\quad Z\one=m\one.
  \end{aligned}
  \tag{SDP}\label{eq:sdp}
\end{equation*}
Here \(\mathbb S^n\) is the space of real symmetric \(n\times n\)
matrices, \(Z\geq0\) denotes entrywise nonnegativity, and
\(\langle A,Z\rangle=\Tr(A^{\mathsf T}Z)\).
The vector \(\one\) has all entries equal to one, and
\(J=\one\one^{\mathsf T}\).
Write \(\mathcal F\) for the feasible set.
We call the SDP tight when \(Z^*\) is its unique optimizer.
In each failure region identified below, we show that \(Z^*\)
is not optimal by constructing a feasible matrix with strictly
larger objective.

We measure the SDP's behavior through the normalized matrix error,
optimal objective gain, and planted overlap:
\begin{equation}
  \mathcal E_m(Z)=\frac{\|Z-Z^*\|_{1,\mathrm{entry}}}{2nm},\qquad
  \mathcal G_m=\frac{V_m-\langle A,Z^*\rangle}{nL},\qquad
  \theta_m(Z)=\frac{\langle Z^*,Z\rangle}{nm}.
  \label{eq:error-metrics}
\end{equation}
Here \(\|\cdot\|_{1,\mathrm{entry}}\) denotes the sum of the
absolute values of all matrix entries.
Positive semidefiniteness and the unit diagonal, together with
entrywise nonnegativity, imply \(0\leq Z_{ij}\leq1\) for every
\(Z\in\mathcal F\). The row constraints then ensure that the
total missing mass within communities equals the total mass
between communities, giving
$
  \mathcal E_m(Z)=1-\theta_m(Z).
$
Thus \(\mathcal E_m\) measures error in the SDP matrix before
labels are extracted. Unless specified otherwise, matrix norms
are operator norms, and ``with high probability'' means with
probability tending to one as \(m\to\infty\) at fixed
\(\alpha,\beta,\gamma\).

To describe the local fluctuations that govern tightness, let
the degree of vertex \(i\) into community~\(C_s\), its margin
against a competing community, and its smallest such margin be
\begin{equation}
  d_{i,s}=\sum_{j\in C_s}A_{ij},\qquad
  D_{i,s}=d_{i,c(i)}-d_{i,s}\quad(s\ne c(i)),\qquad
  \Delta_i=\min_{s\ne c(i)}D_{i,s}.
  \label{eq:block-degrees}
\end{equation}
Let
\begin{equation}
  a=\alpha-\beta,\qquad \sigma=\sqrt{\beta\gamma}.
  \label{eq:a-sigma}
\end{equation}
The typical degree margin is \(aL\), while the centered spectral
noise has scale \(2\sigma L\).
For \(x\leq a\), define the degree-margin lower-tail rate as
\begin{equation}
  I_{\alpha,\beta}(x)=\alpha+\beta-\sqrt{x^2+4\alpha\beta}
       +x\log\frac{x+\sqrt{x^2+4\alpha\beta}}{2\alpha}.
  \label{eq:rate-function}
\end{equation}
Lemma~\ref{lem:margin-ldp} shows that, for each fixed \(x<a\),
\[
  \Prob(D_{i,s}\leq xL)=m^{-I_{\alpha,\beta}(x)+o(1)}.
\]
The rate function is strictly decreasing for \(x<a\), with
\(I_{\alpha,\beta}(a)=0\) and
\(I_{\alpha,\beta}(0)=(\sqrt\alpha-\sqrt\beta)^2\).
We abbreviate it as \(I(x)\) when the parameters are clear.

Finally, for \(a\geq2\sigma\), let the local correction be
\begin{equation}
  \tau=\frac{a-\sqrt{a^2-4\sigma^2}}2,\qquad
  \tau(a-\tau)=\sigma^2.
  \label{eq:tau}
\end{equation}
This correction accounts for the effect of the surrounding
spectral noise on a vertex's degree margin and determines
the boundary and error exponent stated below.

\subsection{Tightness and the local error exponent}

Our first result gives sharp conditions for the planted matrix
to be the unique SDP solution. Away from the stated critical cases, tightness requires both sufficient
signal relative to spectral noise and sufficiently large degree
margins, expressed by \(a>2\sigma\) and
\(I_{\alpha,\beta}(\tau)>1\), respectively. The theorem leaves
the specified critical cases unresolved.

\begin{theorem}[Sharp tightness boundary]
\label{thm:main}
Under \eqref{eq:k-scaling}--\eqref{eq:model}, \(Z^*\) is the unique
SDP optimizer with high probability if
\begin{equation}
  a>2\sigma,\qquad I_{\alpha,\beta}(\tau)>1.
  \label{eq:success-region}
\end{equation}
It is not optimal with high probability if either
\begin{equation}
  a<2\sigma,
  \label{eq:spectral-failure}
\end{equation}
or
\begin{equation}
  a\geq2\sigma,\qquad I_{\alpha,\beta}(\tau)<1.
  \label{eq:local-failure}
\end{equation}
No assertion is made in either of the remaining critical cases:
\(a=2\sigma\) with \(I_{\alpha,\beta}(\sigma)\geq1\), or
\(a>2\sigma\) with \(I_{\alpha,\beta}(\tau)=1\).
\end{theorem}

When \(a>2\sigma\) but \(I_{\alpha,\beta}(\tau)<1\), the planted
matrix is not optimal, yet near-optimal SDP matrices remain close
to it. The next theorem determines the rate at which their
normalized error vanishes, together with the rate of the
normalized objective gain. The bounds hold simultaneously
for all feasible solutions within an additive one of the optimum,
including every optimal solution.

\begin{theorem}[Matching local error exponent]
\label{thm:error}
Assume \eqref{eq:k-scaling}--\eqref{eq:model},
\(a>2\sigma\), and \(r=I_{\alpha,\beta}(\tau)\in(0,1)\).
For every fixed \(\varepsilon>0\), with high probability,
\begin{equation}
  m^{-r-\varepsilon}\leq\mathcal G_m\leq m^{-r+\varepsilon},
  \label{eq:gain-exponent}
\end{equation}
and simultaneously, for every \(Z\in\mathcal F\) satisfying
\(\langle A,Z\rangle\geq V_m-1\), we have
\begin{equation}
  m^{-r-\varepsilon}\leq\mathcal E_m(Z)\leq m^{-r+\varepsilon}.
  \label{eq:error-exponent}
\end{equation}
\end{theorem}

Thus both matrix error and objective gain have exponent
\(r=I(\tau)\): their logarithms, divided by \(\log m\),
converge in probability to \(-r\), uniformly over the stated
class of matrices for the error guarantee. This exponent
reflects the rarity of vertices whose degree margins fall
below the correction~\(\tau L\). The result determines the
decay rate up to subpolynomial factors; it does not specify
a multiplicative prefactor, an expected risk, or behavior
within a critical window.

The tightness boundary differs from the threshold for exact
label recovery. Information-theoretically, exact label recovery
is possible when \(I_{\alpha,\beta}(0)>1\) and impossible
when \(I_{\alpha,\beta}(0)<1\) under the uniform balanced prior
\citep{ZhangZhou16,AbbeBandeiraHall16}.
Appendix~\ref{app:consequences} gives the precise normalization
and converse. Whenever the local correction is defined,
\(I(\tau)<I(0)\), so the local condition for tightness is
more restrictive than the information-theoretic condition.
Together with the spectral failure condition, this implies
that every fixed \(\beta,\gamma>0\) admits an interval
immediately above the information threshold in which
\(Z^*\) is not optimal (Corollary~\ref{cor:gap}).

\subsection{Optimal value and planted overlap}

The preceding theorem describes the error caused by local
fluctuations when \(a>2\sigma\). We now characterize the leading
optimal value and the overlap with the planted matrix on both
sides of the spectral boundary \(a=2\sigma\).

\begin{proposition}[Optimal value and planted overlap]\label{prop:bulk-value}
Under \eqref{eq:k-scaling}--\eqref{eq:model},
\begin{equation}
  \frac{V_m}{nL}\xrightarrow{\Prob}\beta+\max\{a,2\sigma\}.
  \label{eq:bulk-value}
\end{equation}
For any deterministic \(s_m\geq0\) with \(s_m=o(nL)\),
the following limits hold in probability, uniformly over
all \(Z\in\mathcal F\) satisfying
\(\langle A,Z\rangle\geq V_m-s_m\):
\[
  \theta_m(Z)\longrightarrow
  \begin{cases}
    1,&a>2\sigma,\\
    0,&a<2\sigma.
  \end{cases}
\]
\end{proposition}

When \(a>2\sigma\), both objectives have leading term
\(\alpha nL\), so the gain in Theorem~\ref{thm:error}
is of smaller order. When \(a<2\sigma\),
\(\mathcal G_m\to2\sigma-a>0\) and
\(\mathcal E_m(Z)\to1\) in probability for the near-optimal
solutions above. Vanishing normalized overlap alone does
not rule out label recovery. At \(a=2\sigma\), the proposition
gives the limiting value but does not provide an overlap guarantee.

\section{Establishing the sharp error exponent}
\label{sec:proof-ideas}

We outline the proof of Theorem~\ref{thm:error}.
For the upper bounds, we adjust the dual certificate to control
both matrix error and objective gain. For the matching lower
bounds, we construct a feasible matrix whose objective
improvement forces every near-optimal solution to have
matrix error with the same exponent.
Both arguments depend on the number of vertices with degree
margins near \(\tau L\), yielding the common exponent
\(I(\tau)\). Figure~\ref{fig:mechanism} summarizes the
argument, and we provide detailed proofs in Appendices~\ref{sec:inputs}--\ref{sec:failure}.

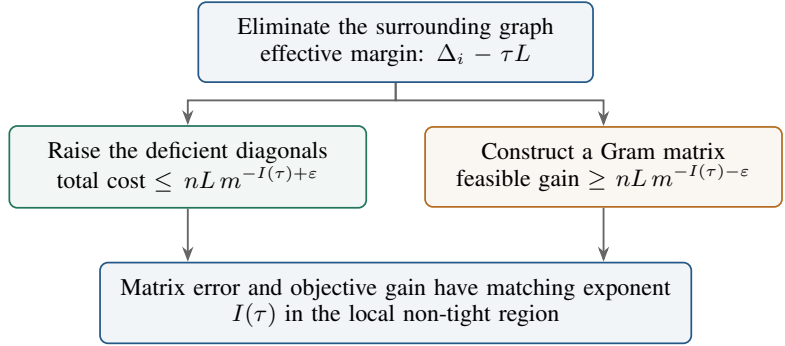
\begin{figure}[ht]
  \centering
  \input{figures/mechanism.tex}
  \caption{The upper and lower bounds in Theorem~\ref{thm:error}.
  To leading order, eliminating the surrounding graph reduces a vertex's
  effective margin by \(\tau L\).
  Repairing the corresponding small diagonals bounds the error from
  above. Many low-margin roots, organized into balanced forests and
  used in a finite-step Gram construction, give a feasible objective
  improvement. The cost and gain have the same exponent because both
  are controlled by the number of vertices with margin near \(\tau L\).}
  \label{fig:mechanism}
\end{figure}

\subsection{The correction to the degree margin}
Let \(W=A-\E A\) and \(X=W/L\). The sharp sparse-matrix
bound of \citet{BenaychGeorgesBordenaveKnowles20} gives
\(\|X\|\leq2\sigma+o_{\Prob}(1)\)
(Lemma~\ref{lem:spectral-edge}).
For \(z>2\sigma\), let
\begin{equation}
  g_z=\frac{z-\sqrt{z^2-4\sigma^2}}{2\sigma^2},\qquad
  \tau_z=\sigma^2g_z.
  \label{eq:gc-tauc}
\end{equation}
Note that when \(a>2\sigma\), we take \(z=a\), and it gives \(\tau_a=\tau\).

To explain the correction, consider a coordinate with margin
\(d\), coupled to the remaining coordinates through a vector
\(b\). If the quadratic form on those remaining coordinates
is represented by a positive definite matrix \(C\), then
$  \min_v\{d-2b^{\mathsf T}v+v^{\mathsf T}Cv\}
  =d-b^{\mathsf T}C^{-1}b.$
Thus optimizing over the remaining coordinates reduces the
effective margin by \(b^{\mathsf T}C^{-1}b\).
When the background diagonal is \(zL\), this reduction is
\(\tau_zL\) to leading order. Estimating this quantity directly
gives a sharper correction than the global norm bound
\(2\sigma L\).

To estimate this reduction at vertices with small
degree margins, fix \(u\in(0,a)\) and let
\begin{equation}
  \mathcal S_u=\{i:\Delta_i\leq uL\},
  \label{eq:marked-set}
\end{equation}
and let \(P_S\) be the coordinate projection onto \(S\).
The following proposition shows that the resolvent, restricted
to these vertices, is asymptotically a scalar multiple of
the identity.

\begin{proposition}[Resolvent on marked vertices]
\label{prop:marked-resolvent}
For every fixed \(u\in(0,a)\) and \(z>2\sigma\),
the inverse \((zI-X)^{-1}\) exists with high probability, and
\begin{equation}
  \left\|P_{\mathcal S_u}
     \bigl[(zI-X)^{-1}-g_zI\bigr]P_{\mathcal S_u}\right\|
  \xrightarrow{\Prob}0.
  \label{eq:marked-resolvent}
\end{equation}
\end{proposition}

Here the main challenge is that \(\mathcal S_u\) depends on \(X\)
and can grow with \(m\). The estimate must therefore account
for this dependence and control the restricted matrix in
operator norm, beyond convergence of individual entries.
We begin with degree-margin estimates conditional on finitely
many edges. For a connected configuration with \(v\) vertices,
\(e\) required edges, and \(j\) marked vertices, the expected
number of embeddings is at most
\(n^vp^e m^{-jI(u)+o(1)}\).
For any fixed radius, we adopt these estimates to show that,
with high probability, neighborhoods of marked vertices
contain no cycles. Taking \(j>1/I(u)\) also bounds the number of marked
vertices in every neighborhood of that radius.

On these tree neighborhoods, we count closed walks to recover
the moments of the semicircle law. We also show that walks
between distinct marked vertices make vanishing contributions.
Because only boundedly many marked vertices lie within any
fixed distance of a given vertex, the sum of the absolute
off-diagonal contributions in each row vanishes uniformly.
After accounting for the centering of \(A\), we obtain
operator-norm convergence for each fixed power of \(X\)
restricted to the marked set. We then use the norm bound
on \(X\) to control the remaining terms in the resolvent
series and conclude \eqref{eq:marked-resolvent}.

\subsection{Upper bounds from the dual certificate}

Let
\begin{equation}
  Q=Z^*/m,\qquad \Pi=I-Q,\qquad \mathcal H=\range\Pi.
  \label{eq:block-projections}
\end{equation}
Here the projection \(Q\) averages coordinates within each planted
community, while \(\Pi\) subtracts these averages.
The dual certificate of
\citet{AgarwalBandeiraKoiliarisKolla15} has slack matrix
\(\mathsf S=\Pi(D^*-A)\Pi\), where \(D^*\) is diagonal with
entries \(D_i^*=\Delta_i-\rho_{c(i)}\) and
\(\max_r|\rho_r|=o_{\Prob}(L)\).
Its multiplier \(\Gamma\) is entrywise nonnegative, and every
feasible \(Z\) satisfies
\begin{equation}
  \langle A,Z^*-Z\rangle
  =\langle\mathsf S,Z\rangle+\langle\Gamma,Z\rangle.
  \label{eq:main-slack-identity}
\end{equation}
This identity remains valid when \(\mathsf S\) is not positive
semidefinite. We can therefore use it to bound the error even
when \(Z^*\) is not optimal, following the arguments in 
\citet[Section~1.1]{FeiChenBayes19}.

We restore positive semidefiniteness by increasing the small
diagonal entries. Fix \(b>\tau\) sufficiently close to \(\tau\),
write \(\widehat d_i=D_i^*+p\), and define a diagonal matrix \(F\) by $
  F_{ii}=(bL-\widehat d_i)_+. $
Thus \(\widehat d_i+F_{ii}\geq bL\) for every vertex.
Choose \(z\in(\max\{b,2\sigma\},a)\) such that \(\tau_z<b\).
The marked resolvent estimate gives a lower bound
\((b-\tau_z-o_{\Prob}(1))L\) on the eigenvalues of the
Schur complement. The background block has eigenvalues
at least \((z-2\sigma-o_{\Prob}(1))L\).
Because the linear map used to eliminate the background
has bounded norm, these two bounds imply
$
  \mathsf S+\Pi F\Pi\succeq\kappa L\Pi $
for some fixed \(\kappa>0\), with high probability.

For any feasible \(Z\),  we have 
$
  \|Z-Z^*\|_{1,\mathrm{entry}}=2m\Tr(\Pi Z)$ and $
  \qquad 0\leq(\Pi Z\Pi)_{ii}\leq2$ by the SDP constraints. 
Combining these relations with
\eqref{eq:main-slack-identity} yields
\begin{equation}
  \langle A,Z-Z^*\rangle
  \leq2\Tr F-\kappa L\Tr(\Pi Z).
  \label{eq:repair-inequality}
\end{equation}
The degree-margin tail estimates bound the total diagonal
adjustment by
\(\Tr F\leq nL\,m^{-I(\tau)+\varepsilon}\)
with high probability, provided \(b\) is sufficiently close
to \(\tau\).

Dropping the nonpositive term in
\eqref{eq:repair-inequality} gives the upper bound on the
objective gain. For any feasible \(Z\) with
\(\langle A,Z\rangle\geq V_m-1\), feasibility of \(Z^*\)
also gives \(\langle A,Z-Z^*\rangle\geq-1\). Hence
\[
  \kappa L\Tr(\Pi Z)\leq2\Tr F+1.
\]
Since \(\mathcal E_m(Z)=\Tr(\Pi Z)/n\), this gives the
matrix-error upper bound uniformly over all such \(Z\).

\subsection{A Gram construction for the lower bound}

To obtain a lower bound on the objective gain, we construct
a feasible matrix that improves on~\(Z^*\).
Writing the matrix as the Gram matrix of unit vectors ensures
positive semidefiniteness and unit diagonal. We preserve the
row sums by arranging the perturbation to have the same
coordinate multiset in each community involved in the
construction. A second Gram matrix then corrects any negative
entries. This construction allows a fixed positive step,
so the objective gain retains the order of the leading term.

Choose an even number \(2\leq K\leq k\) of communities,
which we call active, and index them by \(1,\ldots,K\).
Pair these communities through a perfect matching \(\pi\).
For a vector \(y\) centered within each community, let 
$  \mathcal Q(y)=y^{\mathsf T}Ay-
       \sum_{r\ {\rm active}}\sum_{i\in C_r}
       D_{i,\pi(r)}y_i^2$. 
The following lemma converts a positive value of
\(\mathcal Q(y)\) into a feasible objective improvement.

\begin{lemma}[Finite-step Gram improvement]
\label{lem:finite-gram}
Suppose \(y\) vanishes on inactive communities and has
identical coordinate multisets on the active communities.
Each coordinate in an active community is either
\(-\mu<0\) or positive and at least \(\ell>0\).
Assume \(y\) is centered in every community,
\(\max_i y_i^2=O(1)\), \(K\mu/\ell=o(1)\), and
\(\max_{i,s}d_{i,s}=O(L)\).
If \(\mathcal Q(y)\geq\delta L\|y\|_2^2\) for fixed
\(\delta>0\), then there are constants \(t,c>0\),
independent of \(m\), such that, for all sufficiently
large \(m\), a feasible matrix \(Z(t)\) of rank \(O(k)\)
satisfies
\[
  \langle A,Z(t)-Z^*\rangle\geq cL\|y\|_2^2.
\]
\end{lemma}

To construct \(Z(t)\), take orthonormal vectors
\(e_1,\ldots,e_k,f\). For vertices in active communities, set
\[
  v_i^P=\frac{1+r_i}{2}e_{c(i)}
       +\frac{1-r_i}{2}e_{\pi(c(i))}
       +\sqrt t\,y_i f,
  \qquad r_i=\sqrt{1-2ty_i^2}.
\]
For vertices in inactive communities, set \(v_i^P=e_{c(i)}\).
These vectors have unit norm. Because \(y\) is centered
within each community and has identical coordinate multisets
across active communities, we have $
  \sum_i v_i^P=m\sum_{r=1}^k e_r$. 
Their Gram matrix \(P\) therefore has unit diagonal and
row sums \(m\). Its potentially negative entries
are \(ty_iy_j\) between  active communities
that are not paired by \(\pi\).

To correct these entries, define
\(\bar e=K^{-1}\sum_{r=1}^K e_r\) and form a second
Gram matrix \(G\) using
\[
  v_i^G=\bar e+c_i(e_{c(i)}-\bar e)+\sqrt t\,y_i f,
  \qquad
  c_i=\sqrt{1-\frac K{K-1}ty_i^2}
\]
on the active communities, keeping the inactive vectors
unchanged. The matrix \(G\) also has unit diagonal and
row sums \(m\). The mixture
$Z(t)=(1-\eta)P+\eta G$, where $\eta=2(K-1)\mu/\ell=o(1)$,
has nonnegative entries and hence satisfies all SDP
constraints. Moreover,
\begin{equation}
  \langle A,Z(t)-Z^*\rangle
  \geq t\mathcal Q(y)-C(t^2+\eta t)L\|y\|_2^2.
  \label{eq:finite-gram-gain}
\end{equation}
Choosing a sufficiently small fixed \(t>0\) controls
the quadratic remainder, while \(\eta=o(1)\) makes the
remaining term negligible. This proves the claimed gain.

Appendix~\ref{sec:gram} verifies the constraints and
bounds the remainder. Local Gram constructions for SDP
values also appear in
\citet{MontanariSen16,FanMontanari17}; here we additionally
enforce the row sums and entrywise nonnegativity required
by \eqref{eq:sdp}.

\subsection{Matching lower bounds from low-margin vertices}

We now construct a vector \(y\) that satisfies
Lemma~\ref{lem:finite-gram} and yields the required
objective gain. Take \(K=2\lfloor k/2\rfloor\) active
communities and pair them as above.
Fix \(x<\tau\) with \(I(x)<1\) and
\(\nu\in(0,1-I(x))\).
With high probability, each active community contains at least
$
  M=\lfloor m^{1-I(x)-\nu}\rfloor
$ 
vertices whose margins against its matched community
are at most \(xL\). We use these vertices as roots.
Conditioning on all edges within communities and between
matched communities fixes the roots while leaving edges
between unmatched active communities independent.

We explore the graph from these roots in groups, with one
root from each active community. Reserve all roots before exploration,
and restrict nonroot vertices to satisfy
\(D_{i,\pi(c(i))}\leq(a+L^{-1/8})L\), as in
Lemma~\ref{lem:neutral-forest}. For each group, we grow
trees to a fixed depth \(R\), selecting
\(\lfloor\beta L-L^{3/4}\rfloor\) fresh neighbors in each
active community other than the current vertex's own
community and its matched community.
If an exploration fails, we discard the group and exclude
all vertices it has used from later explorations.
The total number of reserved or previously used vertices in each community
is at most \(ML^{O(R)}=o(m)\).
Consequently, each group has conditional failure probability
\(o(1)\), uniformly over the preceding explorations.
The expected fraction of failed groups is therefore \(o(1)\),
so Markov's inequality shows that at least \(M/2\) groups
succeed with high probability.

Let \(N\) be the number of retained groups and
\(d_L=(\sigma^2+o(1))L^2\) the branching factor.
The construction gives exactly \(Nd_L^\ell\) vertices
at depth \(\ell\) in every active community.
Assign weight \(d_L^{-\ell/2}w_\ell\) to each such vertex,
set all other coordinates to zero, and then subtract the
mean within each community.
The resulting vector \(y\) satisfies the structural
conditions of Lemma~\ref{lem:finite-gram}, with
\(\|y\|_2^2\asymp kN\).
Its quadratic form satisfies
\begin{equation}
  \frac{\mathcal Q(y)}{KNL}
  \geq-xw_0^2
       +2\sigma\sum_{\ell=1}^R w_{\ell-1}w_\ell
       -a\sum_{\ell=1}^R w_\ell^2-o(1).
  \label{eq:main-path-energy}
\end{equation}

Set \(w_0=1\) and optimize over the remaining weights.
The maximum of the leading expression in
\eqref{eq:main-path-energy} is \(\tau_R-x\), where
$
  \tau_0=0$, 
$  \tau_{j+1}=\frac{\sigma^2}{a-\tau_j}$, and
$ \tau_j\uparrow\tau$. 
Since \(x<\tau\), this value is positive for a sufficiently
large fixed \(R\).
Thus the trees reproduce, to any fixed accuracy, the same
correction \(\tau\) that appears in the upper bound.
Lemma~\ref{lem:finite-gram} then gives a feasible objective
gain of at least \(ckNL\), or \(cN/m\) after normalization.
Taking \(x\) sufficiently close to \(\tau\) and then
\(\nu\) sufficiently small proves the objective-gain
lower bound in Theorem~\ref{thm:error}.

To obtain the matrix-error lower bound, we return to
\eqref{eq:main-slack-identity}.
Since \(\|\mathsf S\|=O(L)\) and \(\Gamma\geq0\),
every feasible \(Z\) satisfies
  $\langle A,Z-Z^*\rangle
  \leq CL\Tr(\Pi Z)
  =CnL\,\mathcal E_m(Z)$.
Every solution with objective at least \(V_m-1\) achieves
the optimal gain up to an additive one.
The objective-gain lower bound therefore yields a
matrix-error lower bound with the same exponent.

The proof of Proposition~\ref{prop:bulk-value} uses a
separate construction.
For the upper bound, nonnegativity and the row constraints
bound the noise terms involving \(Q\) by block-degree
deviations whose total contribution is \(o(nL)\).
For the lower bound, we start from a nonnegative polynomial
Gram matrix, center and rescale it, and add \(J/k\) together with a
positive semidefinite correction that restores the unit diagonal.
Appendix~\ref{sec:bulk} shows that the resulting feasible
matrix has noise value approaching \(2\sigma nL\).

\section{Label recovery beyond tightness}
\label{sec:labels}

Even when the planted matrix is not optimal, the SDP output
can provide an initial labeling accurate enough for exact
recovery. We use the following procedure, which requires
only \(A\) and the known community size \(m\) as input.
First, obtain a feasible solution to \eqref{eq:sdp} with
objective at least \(V_m-1\).
Then cluster its rows: among the remaining rows, choose one
whose \(\ell_1\)-ball of radius \(m/2\) contains the most
remaining rows. Retain this row as a center and remove
all remaining rows within distance~\(m\) of it.
Repeat until \(k\) centers have been selected, then assign
every row to its nearest center.
If no rows remain before \(k\) centers have been selected, return arbitrary labels.
Finally, perform \(\lceil\log_2 n\rceil+2\) refinement steps.
At each step, simultaneously assign every vertex to the
label class containing the largest number of its neighbors,
using the labels from the preceding step.
Break all ties deterministically.

\begin{corollary}[Exact labels beyond matrix tightness]
\label{cor:postprocess}
Under \eqref{eq:k-scaling}--\eqref{eq:model}, if
\(a>2\sigma\) and \(I_{\alpha,\beta}(0)>1\), this procedure
recovers all labels up to permutation with high probability.
The conclusion holds simultaneously for every feasible
input SDP matrix with objective at least \(V_m-1\).
\end{corollary}

The proof has two steps. First, row clustering yields
\(o(m)\) mislabeled vertices in total, so the fraction of
errors within every community vanishes.
Second, refinement reduces the number of errors at each
iteration. Under \(I(0)>1\), all true degree margins exceed
\(dL\) with high probability for some fixed \(d>0\).
A vertex mislabeled after an update must therefore have
at least \(dL/2\) neighbors in the preceding error set.
We show that, with high probability, every sufficiently
small vertex set spans too few edges to sustain a large
number of such vertices. This forces the error count to
contract until no errors remain.

The bound holds simultaneously over all sufficiently
small vertex sets, including the error sets produced
by initialization and refinement on the same graph.
Thus the argument requires no independence between these
steps. Appendix~\ref{sec:refinement} gives the full proof
and an additional upper bound on label error with
exponent \(I(0)\).

Row clustering and local refinement are established
methods \citep{FeiChen19,GaoMaZhangZhou17}.
For exact recovery here, a coarser matrix-error bound
suffices: increasing the dual diagonal entries to a level
above \(2\sigma L\) and applying the global norm bound
already yields \(o(m)\) errors after row clustering.
The marked-resolvent estimate is needed to sharpen the
matrix-error bound to the matching exponent \(I(\tau)\).

\section{Numerical illustrations of the asymptotic formulas}
\label{sec:numerics}

Figure~\ref{fig:phase-diagram} illustrates the theoretical
regimes by evaluating their defining formulas.
Along the cut \((\alpha,\beta)=(4.25,1)\),
\(I(0)\approx1.1269>1\), so exact label recovery is
information-theoretically possible for every fixed
\(\gamma>0\).
The tightness boundary occurs at
\(\gamma\approx0.55515\), and the bulk boundary at
\(\gamma=a^2/(4\beta)=2.640625\).
The table shows representative values in the three regimes.

\begin{center}
\begin{tabular}{@{}cccc@{}}
\toprule
\(\gamma\) & \(I(\tau)\) & Limit of \(V_m/(nL)\)
  & SDP matrix behavior\\
\midrule
0.3 & 1.0592 & 4.2500 & Tight\\
1.0 & 0.8923 & 4.2500 & Nonoptimal \(Z^*\), vanishing error\\
2.4 & 0.4613 & 4.2500 & Nonoptimal \(Z^*\), vanishing error\\
3.0 & --- & 4.4641 & Overlap tends to zero\\
\bottomrule
\end{tabular}
\end{center}

At \(\gamma=1\), both the matrix error and the normalized
objective gain have exponent approximately \(0.8923\),
yet the postprocessing procedure in
Corollary~\ref{cor:postprocess} recovers all labels.
At \(\gamma=3\), the normalized objective gain instead
converges to approximately \(0.2141\), and the planted
overlap tends to zero. Here \(\tau\) is not real, so the
local error exponent does not apply.

The accompanying code evaluates these formulas and
locates the boundaries by scalar root finding, with
residual checks. These are illustrations of the asymptotic
results; we do not report finite-sample SDP experiments.
Such experiments would require measuring matrix error,
objective gain, and label error on the same graphs,
with optimization and feasibility errors controlled
below the scales being measured.

\section{Discussion}
\label{sec:discussion}

Our results distinguish the loss of matrix tightness from
the loss of information needed for label recovery.
At the local tightness boundary, vertices with small degree
margins begin to change the SDP optimum. Their frequency
determines the matrix-error exponent, yet postprocessing
still recovers all labels whenever \(a>2\sigma\) and
\(I(0)>1\).
In the bulk regime \(a<2\sigma\), the unnormalized objective gain
is of order \(nL\), and the normalized planted overlap
tends to zero.

The analysis applies to the balanced SDP in \eqref{eq:sdp}
at fixed \(\alpha,\beta,\gamma\).
The remaining critical cases, multiplicative error
prefactors, and finite-sample behavior remain open.
Whether the SDP output permits label recovery in the
bulk regime is also unresolved by our overlap result.
The bulk boundary describes this particular relaxation
and does not establish a computational lower bound
for community recovery.
The growing-community KS results of
\citet{ChinMosselSohnWein26} concern different priors
and recovery criteria; Appendix~\ref{app:comparison}
details the comparison.
\par
\label{end:main-text}

%% file: figures/mechanism.tex
\begin{tikzpicture}[
  font=\small,
  box/.style={draw=#1,rounded corners=3pt,fill=#1!6,
    align=center,inner sep=6pt,line width=.65pt},
  arr/.style={-{Stealth[length=2mm]},line width=.7pt,draw=black!60}
]
\definecolor{proofblue}{HTML}{245887}
\definecolor{prooforange}{HTML}{B46B20}
\definecolor{proofgreen}{HTML}{24745E}
\node[box=proofblue,text width=4.8cm] (response) at (0,0)
  {Eliminate the surrounding graph\\
   effective margin: \(\Delta_i-\tau L\)};
\node[box=proofgreen,text width=4.3cm] (upper) at (-2.75,-1.65)
  {Raise the deficient diagonals\\
   total cost \(\leq nL\,m^{-I(\tau)+\varepsilon}\)};
\node[box=prooforange,text width=4.3cm] (lower) at (2.75,-1.65)
  {Construct a Gram matrix\\
   feasible gain \(\geq nL\,m^{-I(\tau)-\varepsilon}\)};
\node[box=proofblue,text width=7.4cm] (result) at (0,-3.45)
  {Matrix error and objective gain have matching exponent\\
   \(I(\tau)\) in the local non-tight region};
\draw[arr] (response.south) -- ++(0,-.23) -| (upper.north);
\draw[arr] (response.south) -- ++(0,-.23) -| (lower.north);
\draw[arr] (upper.south) -- (upper.south |- result.north);
\draw[arr] (lower.south) -- (lower.south |- result.north);
\end{tikzpicture}

%% file: appendix.tex
\input{appendices/probability}
\input{appendices/upper}
\input{appendices/error-upper}
\input{appendices/gram}
\input{appendices/forests}
\input{appendices/bulk}
\FloatBarrier
\input{appendices/refinement}
\input{appendices/consequences}
\input{appendices/comparison}

%% file: appendices/probability.tex
\section{Probability and random-matrix inputs}
\label{sec:inputs}

We begin with lower-tail estimates for the degree margins. We use these
estimates to count low-margin vertices for the lower bound and to control
their joint occurrence in the resolvent proof. The estimates remain valid
after conditioning on a fixed number of edges. We then give the spectral
norm and degree bounds used in both arguments.

\begin{lemma}[Degree-margin deviations]
\label{lem:margin-ldp}
For every fixed \(x<a\), uniformly over \(i\in C_r\) and \(s\ne r\),
\begin{equation}
  \Prob(D_{i,s}\leq xL)=m^{-I_{\alpha,\beta}(x)+o(1)}.
  \label{eq:margin-tail}
\end{equation}
For each fixed positive integer \(j\), uniformly over distinct vertices
\(i_1,\ldots,i_j\),
\begin{equation}
  \Prob\{\Delta_{i_\nu}\leq xL\text{ for all }1\leq\nu\leq j\}
  \leq m^{-jI_{\alpha,\beta}(x)+o(1)}.
  \label{eq:joint-margin-tail}
\end{equation}
Both estimates hold uniformly conditional on any fixed number of
specified edge values.
\end{lemma}

\begin{proof}
For a fixed vertex and competing community, \(D_{i,s}\) has the law of
\(\operatorname{Bin}(m-1,p)-\operatorname{Bin}(m,q)\), with independent
terms.  For fixed \(\theta\leq0\),
\[
  \frac1L\log\E e^{\theta D_{i,s}}
  =\alpha(e^\theta-1)+\beta(e^{-\theta}-1)+o(1).
\]
The optimizing parameter for the lower tail is
\[
  \theta_x=\log\frac{x+\sqrt{x^2+4\alpha\beta}}{2\alpha}<0.
\]
Chernoff's bound gives the upper bound in \eqref{eq:margin-tail}.
For the lower bound, consider the exponentially tilted distribution with
parameter \(\theta_x\). Under this law, the mean is \(xL+o(1)\)
and the variance is
\(L\sqrt{x^2+4\alpha\beta}+o(L)\).
The central limit theorem for the bounded independent summands gives
a positive limiting probability to the interval
\([xL-C\sqrt L,xL]\) for fixed \(C>0\).
Changing back to the original law gives a lower bound of the form
\(\exp\{-LI(x)+O(\sqrt L)+o(L)\}\), as required.

Conditioning on a bounded number of edge values removes only \(O(1)\)
random summands and shifts the threshold by \(O(1)\). The preceding
estimates are uniform under these changes. For the joint bound, first fix
the competing community for each vertex, and then condition on all edges
among \(i_1,\ldots,i_j\). The remaining incident edges are independent
across these vertices, so the conditional probability is bounded by the
product of the individual bounds. A union over the competing communities
adds a factor of at most \(k^j=L^{O(1)}\).
\end{proof}

\begin{lemma}[Low-margin counts]
\label{lem:margin-counts}
For fixed \(x<a\) with \(I_{\alpha,\beta}(x)<1\), let
\[
  N_{rs}(x)=\sum_{i\in C_r}\one_{\{D_{i,s}\leq xL\}},
  \qquad
  \nu_m(x)=\Prob(D_{i,s}\leq xL)
  \quad(r\ne s).
\]
Then
\begin{equation}
  \max_{r\ne s}\left|\frac{N_{rs}(x)}{m\nu_m(x)}-1\right|
  \xrightarrow{\Prob}0.
  \label{eq:uniform-margin-counts}
\end{equation}
In particular, uniformly over the ordered block pairs,
\(N_{rs}(x)=m^{1-I_{\alpha,\beta}(x)+o_{\Prob}(1)}\).
\end{lemma}

\begin{proof}
For distinct \(i,j\in C_r\), the two margin events share only the
within-block edge \(A_{ij}\). Conditional on that edge, the events are
independent. Let \(u_0\geq u_1\) be their common conditional probabilities
when the edge is absent or present, respectively. Since
\(\nu_m=(1-p)u_0+pu_1\),
\[
  \operatorname{Cov}
    (\one_{\{D_{i,s}\leq xL\}},\one_{\{D_{j,s}\leq xL\}})
  =p(1-p)(u_0-u_1)^2
  \leq\frac{p}{1-p}\nu_m^2.
\]
Consequently,
\begin{equation}
  \frac{\operatorname{Var}N_{rs}(x)}{(m\nu_m(x))^2}
  \leq\frac1{m\nu_m(x)}+\frac p{1-p}.
  \label{eq:margin-count-variance}
\end{equation}
By Lemma~\ref{lem:margin-ldp}, the right-hand side is
\(m^{-(1-I(x))+o(1)}+O(L/m)\).
Chebyshev's inequality and a union bound over \(k(k-1)=O(L^2)\)
ordered pairs prove \eqref{eq:uniform-margin-counts}.
\end{proof}

We next bound the norm of the centered adjacency matrix. Set
\begin{equation}
  W=A-\E A,\qquad X=W/L.
  \label{eq:centered-matrix}
\end{equation} 

\begin{lemma}[Centered norm and degree bounds]
\label{lem:spectral-edge}
Under \eqref{eq:k-scaling}--\eqref{eq:model},
\begin{equation}
  \|X\|\leq2\sigma+o_{\Prob}(1).
  \label{eq:norm-edge}
\end{equation}
Moreover, with probability tending to one, uniformly over vertices,
\begin{equation}
  \deg(i)=(\sigma^2+o(1))L^2,
  \label{eq:uniform-degrees}
\end{equation}
and uniformly over vertex--community pairs and block pairs,
\begin{equation}
  d_{i,s}=O(L),\qquad
  \frac1m\sum_{i\in C_r,j\in C_s}A_{ij}=O(L).
  \label{eq:uniform-block-bounds}
\end{equation}
\end{lemma}

\begin{proof}
The common row-variance sum is
\[
  v_m=(m-1)p(1-p)+(k-1)mq(1-q)
      =(\sigma^2+o(1))L^2.
\]
Apply \citet[Assumption~2.4 and Theorem~2.7]{BenaychGeorgesBordenaveKnowles20}
to \(\widehat W=W/\sqrt{v_m}\).  Its row-variance sums equal one,
its entry variances are at most \(\kappa/n\) for a fixed
\(\kappa\geq1\), and its entries have magnitude at most
\(1/q_{\mathrm{sp}}\), where \(q_{\mathrm{sp}}=\sqrt{v_m}=\Theta(L)\).
The required inequality
\(1\vee q_{\mathrm{sp}}\leq n^{1/10}\kappa^{-1/9}\) holds for large
\(m\), and \(\sqrt{\log n}/q_{\mathrm{sp}}=o(1)\).
The theorem therefore gives \(\E\|\widehat W\|\leq2+o(1)\).
The concentration inequality in \citet[equation~(2.4)]{BenaychGeorgesBordenaveKnowles20}
gives, for every \(u\geq0\),
\[
  \Prob\{|\|\widehat W\|-\E\|\widehat W\||\geq u\}
  \leq2\exp(-c q_{\mathrm{sp}}^2u^2).
\]
Taking \(u=L^{-1/4}\) and using \(\sqrt{v_m}/L\to\sigma\)
proves \eqref{eq:norm-edge}.

The total degree has mean \((\sigma^2+o(1))L^2\).
A Chernoff bound with relative error \(L^{-1/8}\), followed by a
union bound over \(n\) vertices, proves \eqref{eq:uniform-degrees}.
A Chernoff bound for each \(d_{i,s}\), followed by a union over the
\(nk\) vertex--community pairs, bounds all such degrees by a sufficiently
large fixed multiple of \(L\). Averaging within each community gives
the block-pair bound.
\end{proof}

%% file: appendices/upper.tex
\section{Dual analysis and upper bounds}
\label{sec:success}

The dual certificate reduces unique optimality to positivity on the
space \(\cH\) of vectors centered within each community. When \(a>2\sigma\), most of the
diagonal entries exceed the noise norm. We isolate the smaller entries
and use the marked resolvent to account for their interaction with the
rest of the graph.

We use the candidate dual diagonal from
\citet[Section~3]{AgarwalBandeiraKoiliarisKolla15}.
Write \(E_{rs}=\sum_{i\in C_r,j\in C_s}A_{ij}\), and set
\begin{equation}
  M_i=\max_{s\ne c(i)}d_{i,s},\qquad
  \overline M_r=\frac1m\sum_{i\in C_r}M_i,\qquad
  c_0=\min_{r,s}\frac{E_{rs}}m.
  \label{eq:dual-degrees}
\end{equation}
Define the block corrections and the candidate diagonal by
\begin{equation}
  \rho_r=\overline M_r-c_0,\qquad
  D_i^*=\Delta_i-\rho_{c(i)}
       =d_{i,c(i)}-M_i-\overline M_{c(i)}+c_0,
  \qquad D^*=\diag(D_i^*).
  \label{eq:dual-diagonal}
\end{equation}
The minimum defining \(c_0\) includes diagonal block pairs.
Lemma~\ref{lem:block-correction} shows that it is attained across
distinct blocks with high probability.

\begin{lemma}[Dual reduction]
\label{lem:dual-reduction}
If \(x^{\mathsf T}(D^*-A)x>0\) for every nonzero \(x\in\cH\),
then \(Z^*\) is the unique optimal solution of \eqref{eq:sdp}.
\end{lemma}

\begin{proof}
We display the certificate of
\citet[Section~3 and Appendix~A.6]{AgarwalBandeiraKoiliarisKolla15}
in our notation.  Set \(v_i=(M_i-c_0/2)/m\) and
\begin{equation}
  \Gamma_{ij}=
  \begin{cases}
    0,&c(i)=c(j),\\[2pt]
    \displaystyle
    \frac{M_i-d_{i,c(j)}+M_j-d_{j,c(i)}
      +E_{c(i),c(j)}/m-c_0}{m},&c(i)\ne c(j).
  \end{cases}
  \label{eq:explicit-multiplier}
\end{equation}
The numerator is a sum of nonnegative terms.  For \(i\in C_r\)
and \(s\ne r\), its block row sum is
\[
  \sum_{j\in C_s}\Gamma_{ij}
  =M_i-d_{i,s}+\overline M_s-c_0.
\]
Together with \eqref{eq:dual-diagonal}, this shows that the symmetric
slack
\begin{equation}
  \mathsf S=D^*+v\one^{\mathsf T}+\one v^{\mathsf T}-A-\Gamma
  \label{eq:explicit-slack}
\end{equation}
annihilates every planted community indicator.
Each cross block of \(\Gamma\) is a sum of a row term and a column
term, so \(\Pi\Gamma\Pi=0\).  Consequently,
\[
  \Pi\mathsf S\Pi=\Pi(D^*-A)\Pi.
\]
Under the hypothesis, \(\mathsf S\succeq0\) and
\(\ker\mathsf S=\range Q\).  For any feasible \(Z\), the affine
constraints and \(\mathsf SZ^*=0\) give
\[
  \langle A,Z^*-Z\rangle
  =\langle\mathsf S,Z\rangle+\langle\Gamma,Z\rangle\geq0.
\]
If equality holds, positivity forces
\(\range Z\subseteq\ker\mathsf S\), so \(Z\) is block-constant.
Its unit diagonal makes every diagonal block all ones, using the entire
row mass \(m\).  Entrywise nonnegativity then makes all cross blocks
zero.  Thus \(Z=Z^*\).
\end{proof}

To apply the certificate, we must estimate \(D_i^*\).
The next lemma shows that the block corrections are \(o_{\Prob}(L)\);
the degree margins therefore give the leading diagonal values.

\begin{lemma}[Block correction]
\label{lem:block-correction}
The corrections satisfy
\begin{equation}
  \max_{r\in[k]}|\rho_r|=o_{\Prob}(L).
  \label{eq:rho-small}
\end{equation}
\end{lemma}

\begin{proof}
For a fixed vertex, \(M_i\) is the maximum of \(k-1\) independent
\(\operatorname{Bin}(m,q)\) variables, each with mean \(\mu_0=\beta L\).
Bernstein's inequality and a union bound give
\begin{equation}
  \Prob(M_i-\mu_0\geq u)
  \leq (k-1)\exp\!\left(-\frac{u^2}{2(\mu_0+u/3)}\right),
  \qquad u\geq0.
  \label{eq:max-degree-tail}
\end{equation}
In particular, at
\(u=C\{\sqrt{\mu_0(\log k+t)}+\log k+t\}\), the probability
is at most \(e^{-t}\), for a sufficiently large fixed \(C\).
Integrating this tail gives
\begin{align}
  0\leq\E M_i-\mu_0
    &\leq C\{\sqrt{\mu_0\log k}+\log k\}=o(L), \label{eq:max-degree-mean}\\
  \E(M_i-\mu_0)_+^2
    &\leq C\{\mu_0\log k+(\log k)^2\},\qquad
      \E M_i^2=O(L^2). \label{eq:max-degree-second-moment}
\end{align}
The lower bound in the first line follows because the maximum dominates
each individual variable. We use the second moment to average within
blocks.

For distinct vertices in the same planted block, the collections of
cross-block edges are disjoint and hence independent, even when two
edges share an endpoint outside that block.  Therefore
\(\operatorname{Var}\overline M_r=O(L^2/m)\).
For every fixed \(\varepsilon>0\), Chebyshev's inequality gives
\[
  \Prob\left\{\max_r|\overline M_r-\E M_i|>\varepsilon L\right\}
  \leq\frac{Ck}{m\varepsilon^2}=o(1).
\]
For \(r\ne s\), \(E_{rs}\) has law \(\operatorname{Bin}(m^2,q)\);
for \(r=s\), it has law
\(2\operatorname{Bin}(m(m-1)/2,p)\).
Thus \(\operatorname{Var}(E_{rs}/m)=O(L/m)\), and their means
after division by \(m\) are respectively \(\beta L\) and
\(\alpha L-p\).
A union bound over the \(O(k^2)\) block pairs yields
\[
  \max_{r\ne s}|E_{rs}/m-\beta L|=o_{\Prob}(L),\qquad
  \max_r|E_{rr}/m-\alpha L|=o_{\Prob}(L).
\]
Since \(\alpha>\beta\), \(c_0=\beta L+o_{\Prob}(L)\) and its
minimum is attained across blocks with high probability.
Combining this with \eqref{eq:max-degree-mean} proves
\eqref{eq:rho-small}.
\end{proof}

\subsection{The marked-resolvent estimate}
\label{sec:resolvent}

Recall the marked set \(\mathcal S_t=\{i:\Delta_i\leq tL\}\),
the matrix \(X=(A-\E A)/L\), and the semicircle transform \(g_z\)
in \eqref{eq:gc-tauc}.

\begin{proof}[Proof of Proposition~\ref{prop:marked-resolvent}]
Fix \(t\in(0,a)\) and \(z>2\sigma\). All neighborhood radii and
moment orders in the proof are fixed before
letting \(m\to\infty\).
The main step is to show that fixed moments, compressed to the marked
set, approach scalar semicircle moments in operator norm. A norm bound
then allows us to sum the resolvent series.

\emph{Sparsity of marked neighborhoods.}
Lemma~\ref{lem:margin-ldp}, together with a union over the \(k-1\)
wrong communities and the lower bound from a single comparison, gives
\begin{equation}
  \Prob(i\in\cS_t)=m^{-I_{\alpha,\beta}(t)+o(1)}.
  \label{eq:mark-density}
\end{equation}
Since \(I(t)>0\), Markov's inequality gives some fixed \(\eta>0\)
such that, with high probability,
\begin{equation}
  |\cS_t|\leq nm^{-\eta}.
  \label{eq:marked-size}
\end{equation}
The size bound alone does not control how marks cluster near one
another. A fixed connected graph configuration with \(v\) vertices,
\(e\) required present edges, and \(j\) distinct marked vertices has
expected number of embeddings at most
\[
  n^v p^e\,m^{-jI(t)+o(1)}.
\]
Indeed, condition on its required edges and apply the finite-edge
conditional version of Lemma~\ref{lem:margin-ldp}.  The bound is uniform
over the embeddings and community assignments.

Write \(B_R(i)\) for the induced graph on vertices at distance at most
\(R\) from \(i\).  If it contains a cycle, two breadth-first paths
from \(i\), together with an extra edge, contain a connected unicyclic
configuration with at most \(2R+1\) edges.  Such a configuration has
\(e=v\).  Summing the preceding bound over its finitely many shapes,
with \(i\) marked, gives
\begin{equation}
  \E\#\{i\in\cS_t:B_R(i)\text{ contains a cycle}\}
  \leq m^{-I(t)+o(1)}L^{O(R)}=o(1).
  \label{eq:marked-tree-balls}
\end{equation}
If a radius-\(R\) ball contains \(j\) distinct marked vertices,
join its center to them by paths of length at most \(R\).
Their union is connected, has \(v\leq Rj+1\), and satisfies
\(e\geq v-1\).  The expected number of these witnesses is at most
\begin{equation}
  m^{1-jI(t)+o(1)}L^{O(Rj)}.
  \label{eq:bounded-local-marks}
\end{equation}
Choose a fixed integer \(J_t>1/I(t)\).
The bound \eqref{eq:bounded-local-marks} tends to zero with \(j=J_t\).
Consequently, with high probability, every marked radius-\(R\) ball
is a tree and every radius-\(R\) ball contains fewer than \(J_t\)~marks.

\emph{Walk moments.}
Set \(Y=A/L\), and write \(\mu_\ell\) for the semicircle moments:
\begin{equation}
  \mu_{2r}=\frac1{r+1}\binom{2r}{r}\sigma^{2r},
  \qquad \mu_{2r+1}=0.
  \label{eq:semicircle-moments}
\end{equation}
Let
\(\delta_m=\max_i|\deg(i)/L^2-\sigma^2|=o_{\Prob}(1)\).
On the acyclic-neighborhood event, the leading length-\(2r\) closed
walks from a marked vertex traverse \(r\) distinct edges twice.
Their outward and return steps have
\(\frac1{r+1}\binom{2r}{r}\) Catalan patterns.
Each outward choice has \(\deg(v)-O_r(1)\) possibilities.
Walks reusing a previously chosen outward edge have at most \(r-1\)
free edge choices and lose a factor of order \(L^2\).
Uniform degree concentration therefore gives, uniformly over marked \(i\),
\begin{equation}
  (Y^{2r})_{ii}=\mu_{2r}+O_r(\delta_m+L^{-2}),
  \qquad (Y^{2r+1})_{ii}=0.
  \label{eq:diagonal-walk-count}
\end{equation}
This uses a fixed neighborhood radius large enough for the specified
moment order.

For distinct marked vertices \(i,j\) at distance \(h\leq\ell\),
the unique path between them accounts for \(h\) steps of a
length-\(\ell\) walk.  The remaining steps permit at most
\((\ell-h)/2\) free backtracking choices, each with \(O(L^2)\)
possibilities.  Hence
\begin{equation}
  (Y^\ell)_{ij}=O_\ell(L^{-h});
  \label{eq:offdiagonal-walk-count}
\end{equation}
the entry is zero if the distance exceeds \(\ell\).
There are fewer than \(J_t\) marked vertices within distance \(\ell\).
The maximum absolute row-sum bound for the norm of a symmetric matrix
now yields
\begin{equation}
  \left\|P_{\cS_t}Y^\ell P_{\cS_t}-\mu_\ell P_{\cS_t}\right\|
  \leq O_\ell(\delta_m+L^{-1})=o_{\Prob}(1).
  \label{eq:Y-compressed-moments}
\end{equation} 

\emph{Removing the mean matrix.}
The walk count above uses \(Y=A/L\), whose entries are nonnegative.
We now transfer it to \(X\) by showing that mean-matrix terms vanish
after compression to the sparse marked set.
Let \(U\) have the normalized community indicators as columns.
If \(M\) has diagonal \(\alpha\) and off-diagonal entries \(\beta\),
then
\begin{equation}
  \frac{\E A}{L}=UMU^{\mathsf T}-\frac\alpha m I,
  \qquad \|M\|=O(L).
  \label{eq:mean-factorization}
\end{equation}
The maximum-degree event gives \(\|Y\|=O(L)\).
Moreover, each entry of \(Y^rU\) is bounded by
\(L^{C_r}/\sqrt m\) for a constant \(C_r\) depending only on the
fixed order \(r\): sum the nonnegative walk weights in that row.
Thus
\begin{equation}
  \|P_{\cS_t}Y^rU\|_F
  \leq L^{C_r}\sqrt{\frac{k|\cS_t|}{m}}
  \leq L^{C_r}k\,m^{-\eta/2}.
  \label{eq:mean-removal-bound}
\end{equation}
Expand \(X^\ell=(Y-UMU^{\mathsf T}+\alpha I/m)^\ell\) as a sum
of noncommutative words.  In a word containing \(UMU^{\mathsf T}\),
cut at its first and last such factors.  The endpoint factors are of the
form bounded in \eqref{eq:mean-removal-bound}; every intermediate
factor has norm at most a fixed power of \(L\).
Its marked compression consequently has norm at most
\begin{equation}
  L^{C_\ell}\frac{k|\cS_t|}{m}
  \leq L^{C_\ell}k^2m^{-\eta}=o(1),
  \label{eq:mean-word-bound}
\end{equation}
with \(C_\ell\) enlarged if necessary.
A word containing \(\alpha I/m\) but no \(UMU^{\mathsf T}\)
has norm at most \(m^{-1}L^{C_\ell}=o(1)\).
There are only finitely many words at fixed \(\ell\).
Combining these bounds with \eqref{eq:Y-compressed-moments} gives
\begin{equation}
  \left\|P_{\cS_t}X^\ell P_{\cS_t}-\mu_\ell P_{\cS_t}\right\|
  =o_{\Prob}(1).
  \label{eq:X-compressed-moments}
\end{equation} 

\emph{Summing the resolvent series.}
Fix \(b\) with \(2\sigma<b<z\).
By Lemma~\ref{lem:spectral-edge}, \(\|X\|\leq b\) with high
probability.  On that event,
\[
  (zI-X)^{-1}=\sum_{\ell=0}^{\infty}\frac{X^\ell}{z^{\ell+1}},
  \qquad
  \left\|\sum_{\ell>N}\frac{X^\ell}{z^{\ell+1}}\right\|
  \leq\frac{(b/z)^{N+1}}{z-b}.
\]
The corresponding scalar moment tail obeys the same bound because the
semicircle law is supported on \([-2\sigma,2\sigma]\).
First fix \(N\) and let \(m\to\infty\), using
\eqref{eq:X-compressed-moments}; then let \(N\to\infty\).
Since \(\sum_{\ell\geq0}\mu_\ell/z^{\ell+1}=g_z\), this proves
\eqref{eq:marked-resolvent}.
\end{proof}

\subsection{Unique optimality above the tightness boundary}
\begin{proof}[Proof of the success assertion in Theorem~\ref{thm:main}]
\emph{The smallest dual diagonal.}
Assume \(a>2\sigma\) and \(I_{\alpha,\beta}(\tau)>1\).
Choose \(\varepsilon>0\) so that
\begin{equation}
  \tau+3\varepsilon<a,\qquad I_{\alpha,\beta}(\tau+3\varepsilon)>1.
  \label{eq:delta-choice}
\end{equation}
Lemma~\ref{lem:margin-ldp} and a union bound over
\(n(k-1)=mL^{O(1)}\) comparisons imply, with high probability,
\begin{equation}
  \min_i\Delta_i>(\tau+3\varepsilon)L.
  \label{eq:min-delta}
\end{equation}
Lemma~\ref{lem:block-correction} then gives
\begin{equation}
  \min_iD_i^*>(\tau+2\varepsilon)L.
  \label{eq:min-Dstar}
\end{equation}
Choose \(z\in(2\sigma,a)\) sufficiently close to \(a\) that
\begin{equation}
  \tau<\tau_z<\tau+\varepsilon.
  \label{eq:tauc-choice}
\end{equation} 

\emph{The Schur complement.}
Define the shifted diagonal and the small-diagonal set by
\begin{equation}
  \widehat d_i=D_i^*+p,\qquad
  \mathcal S=\{i:\widehat d_i<zL\},\qquad
  \mathcal T=[n]\setminus\mathcal S.
  \label{eq:small-diagonal-set}
\end{equation}
The shift by \(p\) accounts for the zero diagonal of \(A\).
For any fixed \(t\in(z,a)\), \eqref{eq:dual-diagonal} and
\eqref{eq:rho-small} imply, with high probability,
\begin{equation}
  \mathcal S\subseteq\cS_t.
  \label{eq:S-marked-subset}
\end{equation}
Since \(\Pi\E A\Pi=-p\Pi\), it suffices to prove positive definiteness of
\begin{equation}
  M_D=\diag(\widehat d_i)-W.
  \label{eq:MD}
\end{equation}
Lemma~\ref{lem:spectral-edge} makes the comparison
matrix \(M_z=zLI-W\) positive definite with high probability.
Its Schur complement onto \(\mathcal S\) is
\begin{equation}
  C_{\mathcal S}
  =zLI_{\mathcal S}-W_{\mathcal S\mathcal S}
   -W_{\mathcal S\mathcal T}
    (zLI_{\mathcal T}-W_{\mathcal T\mathcal T})^{-1}
    W_{\mathcal T\mathcal S}.
  \label{eq:base-schur}
\end{equation}
If \(\mathcal S\) is empty, \(M_D\succeq M_z\succ0\) and we
are done.  Otherwise, use Proposition~\ref{prop:marked-resolvent}
and \eqref{eq:S-marked-subset}.
For any matrix \(\mathcal R\), the deterministic inequality
\[
  \|P_{\mathcal S}\mathcal R P_{\mathcal S}\|
  \leq\|P_{\cS_t}\mathcal R P_{\cS_t}\|
\]
holds on this containment event, so no further independence is needed.
The inverse of a Schur complement is the corresponding principal
block of the inverse. Applying this identity gives
\begin{equation}
  C_{\mathcal S}^{-1}
  =\frac1L\bigl(g_zI_{\mathcal S}+o_{\Prob}(1)\bigr).
  \label{eq:base-schur-inverse}
\end{equation}
Since \(g_z>0\) and \(g_z^{-1}=z-\sigma^2g_z=z-\tau_z\),
\begin{equation}
  C_{\mathcal S}=(z-\tau_z)LI_{\mathcal S}+o_{\Prob}(L).
  \label{eq:base-schur-value}
\end{equation}
All matrix error terms here are in operator norm.

\emph{Positivity of the dual form.}
The \(\mathcal T\times\mathcal T\) principal block of \(M_D\)
satisfies
\[
  (M_D)_{\mathcal T\mathcal T}
  \succeq zLI_{\mathcal T}-W_{\mathcal T\mathcal T}\succ0.
\]
Inverse monotonicity therefore bounds the Schur complement \(K_D\) by
\begin{equation}
  K_D\succeq C_{\mathcal S}
       +\diag(\widehat d_i-zL:i\in\mathcal S).
  \label{eq:schur-comparison}
\end{equation}
Using \eqref{eq:min-Dstar}, \eqref{eq:tauc-choice}, and
\eqref{eq:base-schur-value},
\begin{equation}
  \lambda_{\min}(K_D)
  \geq\min_{i\in\mathcal S}\widehat d_i-\tau_zL+o_{\Prob}(L)
  \geq(\varepsilon+o_{\Prob}(1))L>0
  \label{eq:schur-positive}
\end{equation}
with high probability.  The \(\mathcal T\)-block is positive definite
as well, so \(M_D\succ0\).

Finally,
\(\E A=(p-q)Z^*+qJ-pI\).
For every nonzero \(x\in\cH\), the first two terms vanish in the
quadratic form, and consequently
\begin{equation}
  x^{\mathsf T}(D^*-A)x=x^{\mathsf T}M_Dx>0.
  \label{eq:slack-on-H}
\end{equation}
Lemma~\ref{lem:dual-reduction} proves unique SDP recovery.
\end{proof}

%% file: appendices/error-upper.tex
\subsection{Matrix-error bounds from the dual certificate}
\label{app:error-upper}

When some margins fall below \(\tau L\), the slack matrix need not be
positive semidefinite. We increase the small entries of the dual diagonal
and use the total adjustment to bound both matrix error and objective gain.
This use of a planted dual certificate beyond tightness follows
\citet[Section~1.1]{FeiChenBayes19}. The marked-resolvent estimate in
Proposition~\ref{prop:marked-resolvent} allows us to make the adjustment
only below a threshold arbitrarily close to \(\tau L\).

\begin{lemma}[Deterministic repair inequality]
\label{lem:repair}
For every \(Z\in\mathcal F\),
\begin{equation}
  \|Z-Z^*\|_{1,\mathrm{entry}}=2m\Tr(\Pi Z),\qquad
  0\leq(\Pi Z\Pi)_{ii}\leq2.
  \label{eq:trace-error-identity}
\end{equation}
If a nonnegative diagonal matrix \(F\) satisfies
\(\mathsf S+\Pi F\Pi\succeq\kappa L\Pi\), with \(\kappa>0\),
then
\begin{equation}
  \langle A,Z-Z^*\rangle
  \leq2\Tr F-\kappa L\Tr(\Pi Z).
  \label{eq:full-repair-inequality}
\end{equation}
In particular, if \(\langle A,Z\rangle\geq\langle A,Z^*\rangle-r_0\)
for \(r_0\geq0\), then
\begin{equation}
  \mathcal E_m(Z)\leq\frac{r_0+2\Tr F}{\kappa nL}.
  \label{eq:repair-error}
\end{equation}
\end{lemma}

\begin{proof}
Positive semidefiniteness and unit diagonal imply
\(|Z_{ij}|\leq1\). Entrywise nonnegativity gives \(0\leq Z_{ij}\leq1\).
The row constraints give total entry mass \(nm\). Thus the missing
within-community mass equals the mass between communities, say \(T\).
Consequently,
\(\|Z-Z^*\|_{1,\mathrm{entry}}=2T\), while
\(\Tr(\Pi Z)=n-\langle Z^*,Z\rangle/m=T/m\).
For \(i\in C_r\),
\[
  (\Pi Z\Pi)_{ii}
  =1-\frac2m\sum_{j\in C_r}Z_{ij}
     +\frac1{m^2}\sum_{j,\ell\in C_r}Z_{j\ell}.
\]
The diagonal entry is nonnegative because \(\Pi Z\Pi\succeq0\).
The middle sum is nonnegative and the final sum is at most \(m^2\),
which gives the upper bound of two.

The construction in Appendix~\ref{sec:success} gives
\(\mathsf S=\Pi\mathsf S\Pi\), \(\Gamma\geq0\), and
the identity \eqref{eq:main-slack-identity}, without a positivity
assumption on \(\mathsf S\). Taking the matrix inner product of
\(\mathsf S+\Pi F\Pi\succeq\kappa L\Pi\) with \(Z\succeq0\) gives
\[
  \kappa L\Tr(\Pi Z)
  \leq\langle\mathsf S,Z\rangle+\langle F,\Pi Z\Pi\rangle
  \leq-\langle A,Z-Z^*\rangle+2\Tr F.
\]
Rearranging proves \eqref{eq:full-repair-inequality}. The objective
assumption then gives \eqref{eq:repair-error}.
\end{proof}

\begin{lemma}[Sharp diagonal repair]
\label{lem:clipped}
Assume \(a>2\sigma\). For every fixed \(b\in(\tau,a)\), the
diagonal matrix defined by
\[
  F_{ii}=(bL-D_i^*-p)_+
\]
satisfies
\(\mathsf S+\Pi F\Pi\succeq\kappa L\Pi\) with high probability
for some fixed \(\kappa>0\). For every fixed \(u\in(b,a)\) and
\(\xi>0\), it also satisfies
\begin{equation}
  \Tr F\leq nL\,m^{-I(u)+\xi}
  \label{eq:repair-trace}
\end{equation}
with high probability.
\end{lemma}

\begin{proof}
Choose a fixed \(z\in(\max\{2\sigma,b\},a)\) sufficiently close to
\(a\) that \(\tau_z<b\), and write
\(\widehat d_i=D_i^*+p\).
Let \(\mathcal S=\{i:\widehat d_i<zL\}\) and
\(\mathcal T=[n]\setminus\mathcal S\).
For fixed \(v\in(z,a)\), the bound
\(\max_r|\rho_r|=o_{\Prob}(L)\)
gives \(\mathcal S\subseteq\mathcal S_v\) with high probability.
As in \eqref{eq:base-schur-value}, the Schur complement of
\(zLI-W\) onto \(\mathcal S\) is
\[
  C_{\mathcal S}=(z-\tau_z)LI_{\mathcal S}+o_{\Prob}(L)
\]
in operator norm. For
\(M=\diag(\max\{\widehat d_i,bL\})-W\), the \(\mathcal T\)-block
has minimum eigenvalue at least \((z-2\sigma-o_{\Prob}(1))L\).
The inverse comparison used in \eqref{eq:schur-comparison} bounds the
Schur complement of \(M\) onto \(\mathcal S\) by
\[
  K_M\succeq C_{\mathcal S}
    +\diag(\max\{\widehat d_i,bL\}-zL:i\in\mathcal S)
  \succeq(b-\tau_z-o_{\Prob}(1))LI_{\mathcal S}.
\]
If \(\mathcal S\) is empty, the lower bound on \(M\) follows directly
from the background block. Otherwise, the two bounds above give a
lower bound on the full matrix as follows. Write
\(M=\left(\begin{smallmatrix}A_0&B_0\\
B_0^{\mathsf T}&C_0\end{smallmatrix}\right)\) in this decomposition.
Here \(C_0\succeq cL I\) and \(K_M\succeq c'L I\).
By Lemma~\ref{lem:spectral-edge}, with high probability,
\(\|B_0\|\leq\|W\|\leq(2\sigma+1)L\).
The identity
\[
  \binom{x}{y}^{\mathsf T}M\binom{x}{y}
  =x^{\mathsf T}K_Mx+
    (y+C_0^{-1}B_0^{\mathsf T}x)^{\mathsf T}
      C_0(y+C_0^{-1}B_0^{\mathsf T}x)
\]
expresses the quadratic form as a sum of two nonnegative terms.
The map \((x,y)\mapsto(x,y+C_0^{-1}B_0^{\mathsf T}x)\) and its
inverse have norms bounded by fixed constants with high probability.
Hence \(M\succeq\kappa LI\)
for a fixed \(\kappa>0\).
Restricting this inequality to \(\mathcal H\) proves the first assertion, since
\(\Pi(D^*-A)\Pi=\Pi(\diag(\widehat d_i)-W)\Pi\).

For the total adjustment, the uniform block-degree bounds and
\eqref{eq:rho-small} give
\(\max_i F_{ii}\leq CL\) with high probability. Since \(u>b\),
\eqref{eq:rho-small} also ensures that \(F_{ii}>0\) implies
\(\Delta_i\leq uL\) with high probability.
The expected number of vertices with \(\Delta_i\leq uL\) is at most
\(n(k-1)m^{-I(u)+o(1)}=n m^{-I(u)+o(1)}\).
Markov's inequality bounds this count by \(n m^{-I(u)+\xi/2}\)
with high probability. Multiplying by \(CL\) and using
\(C\leq m^{\xi/2}\) for large \(m\) proves \eqref{eq:repair-trace}.
\end{proof}

For later use, the uniform degree bounds also give
\begin{equation}
  \|\mathsf S\|
  \leq\|D^*+pI\|+\|W\|=O_{\Prob}(L).
  \label{eq:slack-norm}
\end{equation}
Combining this norm bound with \eqref{eq:main-slack-identity} and
\(\Gamma\geq0\) gives, simultaneously for every feasible \(Z\),
\begin{equation}
  \langle A,Z-Z^*\rangle
  \leq CL\Tr(\Pi Z)=CnL\,\mathcal E_m(Z)
  \label{eq:gain-forces-error}
\end{equation}
on a high-probability event.

\paragraph{Upper bounds in Theorem~\ref{thm:error}.}
Fix \(\varepsilon>0\). By continuity, choose
\(\tau<b<u<a\) sufficiently close to \(\tau\) that
\(I(u)>I(\tau)-\varepsilon/4\), then take \(\xi<\varepsilon/4\).
Lemmas~\ref{lem:repair}--\ref{lem:clipped} bound the optimal objective gain by
\(2\Tr F\). Every feasible matrix with objective at least \(V_m-1\)
has error at most
\((1+2\Tr F)/(\kappa nL)\).
The constants in these bounds are at most \(m^{\varepsilon/2}\)
for all sufficiently large \(m\).
Since \(I(\tau)<1\), \(1/(nL)=o(m^{-I(\tau)})\).
This proves both uniform upper bounds in
\eqref{eq:gain-exponent}--\eqref{eq:error-exponent}.

\paragraph{A coarser bound for initialization.}
If \(b\in(2\sigma,a)\), then
\[
  \diag(\max\{D_i^*+p,bL\})-W
  \succeq(b-2\sigma-o_{\Prob}(1))LI
\]
follows directly from Lemma~\ref{lem:spectral-edge}.
Choose any \(u\in(b,a)\), so that \(I(u)>0\).
Bounding the total diagonal adjustment as above gives, for some fixed
\(c_{\mathrm{init}}>0\),
\begin{equation}
  \sup_{\substack{Z\in\mathcal F\\
        \langle A,Z\rangle\geq V_m-1}}\mathcal E_m(Z)
  \leq m^{-c_{\mathrm{init}}}
  \label{eq:coarse-initialization}
\end{equation}
with high probability. Thus the global norm bound suffices for this
initialization guarantee.

Figure~\ref{fig:margin-rates} compares the three relevant levels of
the rate function. At \((\alpha,\beta,\gamma)=(5,1,3)\), we have
\(a=4>2\sqrt3\) and \(\tau=1\). The information condition
\(I(0)>1\) holds, while \(I(\tau)<1\), so the planted matrix is
not optimal. The coarse argument still gives a positive error exponent,
but its limiting rate \(I(2\sigma)\) is much smaller than
\(I(\tau)\).

\begin{figure}[htbp]
  \centering
  \includegraphics[width=.78\linewidth]{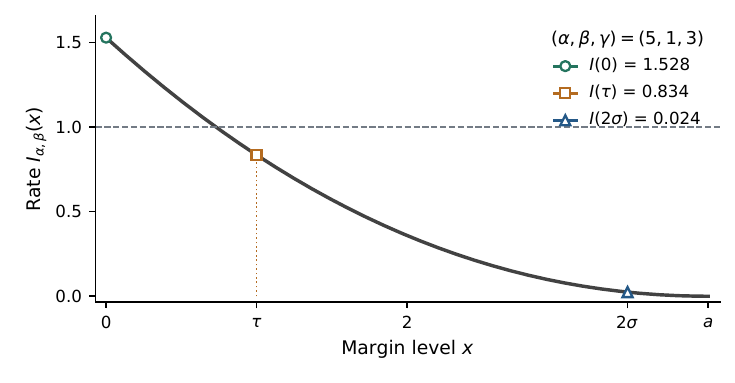}
  \caption{Degree-margin rates at the information, local, and coarse
  spectral thresholds. For \((\alpha,\beta,\gamma)=(5,1,3)\),
  \(I(0)\approx1.528\), \(I(\tau)\approx0.834\), and
  \(I(2\sigma)\approx0.024\). The dashed line marks rate one.
  The local error law uses \(I(\tau)\); the coarser initialization
  bound can approach \(I(2\sigma)\) as its clipping threshold
  decreases to \(2\sigma L\). All points are evaluations of
  \eqref{eq:rate-function}.}
  \label{fig:margin-rates}
\end{figure}

%% file: appendices/gram.tex
\section{Finite-step Gram construction}
\label{sec:gram}

We prove Lemma~\ref{lem:finite-gram} by constructing a feasible matrix
with a step size that does not shrink as \(m\) grows. The identical
coordinate multisets in the active communities preserve the row sums.
We then mix two Gram matrices to make all entries nonnegative and bound
the effect of this adjustment on the objective.

\begin{proof}[Proof of Lemma~\ref{lem:finite-gram}]
Index the active communities by \(1,\ldots,K\).
All margins below are extended by zero on inactive vertices.
Let \(B=\max_i y_i^2\), \(E=\|y\|_2^2\), and choose a fixed
\(D\) with \(d_{i,s}\leq DL\) for all vertex--community pairs.
The assumptions bound \(B\) and \(D\) independently of \(m\).
Write
\[
  \eta=\frac{2(K-1)\mu}{\ell}=o(1),
  \qquad 0<tB\leq\frac18.
\]
For large \(m\), \(0\leq\eta\leq1\).
Use the vectors \(v_i^P,v_i^G\) defined on active communities in
Section~\ref{sec:proof-ideas}, and set
\(v_i^P=v_i^G=e_{c(i)}\) on inactive communities, with these extra
unit vectors orthogonal to the active coordinates and to \(f\).
Let \(P_{ij}=\langle v_i^P,v_j^P\rangle\) and
\(G_{ij}=\langle v_i^G,v_j^G\rangle\).

\emph{Positive semidefiniteness, unit diagonal, and row sums.}
Both matrices are positive semidefinite because they are Gram matrices.
The identities
\[
  \frac{1+r_i^2}{2}+ty_i^2=1,\qquad
  \frac1K+\frac{K-1}{K}c_i^2+ty_i^2=1
\]
give unit diagonal. Since the active communities have identical coordinate
multisets, the sums of \(r_i\) and of \(c_i\) agree across these
communities. The coordinates of \(y\) sum to zero in each community,
so the total coefficient along \(f\) is zero. For the vectors defining
\(P\), the total coefficient along \(e_s\) is
\[
  \frac12\sum_{i\in C_s}(1+r_i)
  +\frac12\sum_{i\in C_{\pi(s)}}(1-r_i)=m.
\]
For the vectors defining \(G\), the equal block sums of \(c_i\)
and the identity \(\sum_s(e_s-\bar e)=0\) give the same conclusion.
Thus the vectors in each construction sum to \(m\sum_{s=1}^k e_s\).
The coefficients of each individual vector along \(e_1,\ldots,e_k\)
sum to one. Taking its inner product with the total vector sum therefore
gives \(P\one=G\one=m\one\).

Set \(Z(t)=(1-\eta)P+\eta G\).
It is the Gram matrix of the direct-sum vectors
\((\sqrt{1-\eta}\,v_i^P,\sqrt\eta\,v_i^G)\), so its rank is at
most \(2(k+1)\). At \(t=0\), \(P=G=Z^*\).

\emph{Entrywise nonnegativity.}
Within an active community,
\[
  P_{ij}=\frac{1+r_ir_j}{2}+ty_iy_j,\qquad
  G_{ij}=\frac1K+\frac{K-1}{K}c_ic_j+ty_iy_j.
\]
For \(u,v\in[0,1]\),
\(\sqrt{1-u}\sqrt{1-v}\geq1-u-v\).
Both entries are therefore at least \(1-3tB\geq5/8\).
For distinct active communities,
\begin{equation}
  Z_{ij}(t)=ty_iy_j+\frac{\eta}{K}(1-c_ic_j)
   +\frac{1-\eta}{2}(1-r_ir_j)
      \mathbf1_{\{c(j)=\pi(c(i))\}}.
  \label{eq:gram-cross-entries}
\end{equation}
The final term is nonnegative. The arithmetic-geometric mean
inequality gives
\[
  \frac{1-c_ic_j}{K}
  \geq\frac{t(y_i^2+y_j^2)}{2(K-1)}.
\]
For equal signs, every term in \eqref{eq:gram-cross-entries} is
nonnegative. When \(y_i\geq\ell\) and \(y_j=-\mu\),
\[
  Z_{ij}(t)\geq-t\mu y_i+\frac{\eta t y_i^2}{2(K-1)}
  =t\mu y_i\left(\frac{y_i}{\ell}-1\right)\geq0.
\]
Entries involving inactive blocks are unchanged from \(Z^*\).
Thus \(Z(t)\in\mathcal F\).

\emph{The objective of the paired construction.}
Let \(A_\pi\) equal \(A\) within active communities, \(-A\)
across matched active communities, and zero elsewhere.
Its active row sums are
\(D_i^\pi=d_{i,c(i)}-d_{i,\pi(c(i))}\).
Write \(b_i=1-r_i\), with \(b_i=0\) on inactive vertices.
Since \(b_i=ty_i^2+b_i^2/2\), direct expansion yields
\begin{align}
  \langle A,P-Z^*\rangle
    &=t y^{\mathsf T}Ay-\sum_iD_i^\pi b_i
      +\tfrac12b^{\mathsf T}A_\pi b \notag\\
    &=t\mathcal Q(y)
      -\frac14\sum_{i,j}(A_\pi)_{ij}(b_i-b_j)^2.
  \label{eq:paired-gram-identity}
\end{align}
Here and below, the edge sums count both orientations.
The absolute row sums of \(A_\pi\) are at most \(2DL\), and
\[
  0\leq b_i=\frac{2ty_i^2}{1+r_i}\leq2ty_i^2,\qquad
  \sum_i b_i^2\leq4t^2BE.
\]
These bounds show that the remainder in
\eqref{eq:paired-gram-identity} has absolute value at most \(8DBt^2LE\).

\emph{The effect of the second Gram matrix.}
Let \(A_{\rm in}\) be the adjacency matrix restricted to edges within active communities,
\(A_{\rm act}\) the full adjacency restricted to active vertices,
and \(T=A_{\rm in}-A_{\rm act}/K\).
Define
\[
  \bar D_i=d_{i,c(i)}
       -\frac1{K-1}
          \sum_{\substack{s\ {\rm active}\\s\ne c(i)}}d_{i,s},
  \qquad h_i=1-c_i.
\]
The identity \(h_i=\frac{K}{2(K-1)}ty_i^2+h_i^2/2\)
gives
\begin{equation}
  \langle A,G-Z^*\rangle
  =t\left(y^{\mathsf T}Ay-\sum_i\bar D_i y_i^2\right)
      -\frac12\sum_{i,j}T_{ij}(h_i-h_j)^2.
  \label{eq:global-gram-identity}
\end{equation}
The absolute row sums of \(T\) are at most \(2DL\);
\(h_i\leq2ty_i^2\), and
\(|\bar D_i-D_i^\pi|\leq DL\).
When we subtract \eqref{eq:paired-gram-identity}, the term
\(t y^{\mathsf T}Ay\) cancels. Bounding the margin difference and
both remainders gives
\[
  |\langle A,G-P\rangle|
  \leq DtLE+24DBt^2LE\leq4DtLE.
\]
Therefore
\begin{equation}
  \langle A,Z(t)-Z^*\rangle
  \geq t\mathcal Q(y)-8DBt^2LE-4D\eta tLE.
  \label{eq:full-gram-gain}
\end{equation}
Choose \(t>0\) fixed with \(tB\leq1/8\) and
\(8DBt\leq\delta/4\), then take \(m\) sufficiently large that
\(4D\eta\leq\delta/4\). The gain is at least
\(\delta tLE/2\), proving the lemma.
\end{proof}

The step size depends on \(\delta\), \(B\), and \(D\).
It is fixed as \(m\to\infty\), but may decrease when the energy
margin \(\delta\) approaches zero.

%% file: appendices/forests.tex
\section{Multi-root forests and matching lower bounds}
\label{sec:failure}

The Gram construction requires a vector with identical coordinate
multisets in the active communities. We obtain such a vector from
forests rooted at low-margin vertices. We explore the roots in groups,
with one root from each active community, and keep a group only if every
exploration succeeds. Each retained group then has the same number of
vertices at each depth in every active community. We show that the
discarded groups form a vanishing fraction, leaving enough roots for
the objective-gain lower bound.

Let \(K=2\lfloor k/2\rfloor\), take the first \(K\) communities
as active, and fix a perfect matching \(\pi\) before the graph is sampled.
Write \(D_i^\pi=D_{i,\pi(c(i))}\) on active vertices and
\(D_i^\pi=0\) elsewhere.
For \(r\) active, write
\(\mathcal N(r)=\{1,\ldots,K\}\setminus\{r,\pi(r)\}\).
Let \(\mathcal F_{\rm match}\) be the sigma-algebra generated by all
edges within active communities and between matched active communities.
It determines every matched margin \(D_{i,\pi(c(i))}\).
Conditional on \(\mathcal F_{\rm match}\), edges between unmatched
active communities remain independent. Set
\begin{equation}
  r_L=\lfloor\beta L-L^{3/4}\rfloor,\qquad
  d_L=(K-2)r_L=(\sigma^2+o(1))L^2.
  \label{eq:forest-degree}
\end{equation}
Both integers are positive for large \(m\).

\begin{lemma}[Balanced forests from many roots]
\label{lem:neutral-forest}
Fix \(R\geq1\) and \(\zeta>0\). Let \(1\leq M\leq m^{1-\zeta}\)
be an integer, and choose \(M\) distinct roots in each active block
by an \(\mathcal F_{\rm match}\)-measurable rule.
With high probability, there are \(N\geq M/2\) groups of roots with
pairwise disjoint depth-\(R\) forests such that every parent has
\(r_L\) children in each active community other than its own and its
matched community.
In their union, each active block contains exactly
\(Nd_L^\ell\) vertices at depth \(\ell\).
Every nonroot vertex satisfies
\[
  D_{i,\pi(c(i))}\leq(a+L^{-1/8})L.
\]
\end{lemma}

\begin{proof}
Set \(\varepsilon_L=L^{-1/8}\) and
\[
  B_r=\{i\in C_r:D_{i,\pi(r)}>(a+\varepsilon_L)L\},
  \qquad G_r=C_r\setminus B_r.
\]
A Bernstein bound gives
\(\Prob(i\in B_r)\leq e^{-cL^{3/4}}\).
Markov's inequality followed by a union over the \(K=O(L)\)
blocks gives the \(\mathcal F_{\rm match}\)-measurable event
\begin{equation}
  \max_r|B_r|/m\leq e^{-c'L^{3/4}}
  \label{eq:good-matched-pools}
\end{equation}
with high probability.

Condition on any realization of \(\mathcal F_{\rm match}\) satisfying
\eqref{eq:good-matched-pools}.
List the roots in each block in a deterministic order.
Group \(j\) consists of the \(j\)-th root in every active block.
Reserve all roots before starting the exploration, so that no root can
be selected as a child. Process the groups successively and the vertices
within each group in breadth-first order. For a parent \(i\in C_r\),
visit the communities in \(\mathcal N(r)\) in label order. Query all
edges to the currently unused pool \(U\subseteq G_s\), select the
first \(r_L\) neighbors, and mark them as used. If there are too few
neighbors, discard the entire group. All vertices already used by that
group remain unavailable to later explorations.

None of the edges queried in a given step has been exposed earlier.
The candidate vertices have not yet been processed as parents, and any
earlier queries from the current parent went to other communities.
Queries from previous parents concern different edges. Moreover, these
edges lie outside \(\mathcal F_{\rm match}\). They are therefore
conditionally independent Bernoulli variables with parameter \(q\).

Even counting all reserved roots and all discarded partial groups, the
total number of used vertices is at most
\(MK\sum_{\ell=0}^R d_L^\ell=M L^{O(R)}\).
Uniformly over all earlier exploration histories,
\[
  |U|\geq m(1-e^{-c'L^{3/4}})-M L^{O(R)},\qquad
  q|U|=\beta L-o(1).
\]
The second relation uses \(M\leq m^{1-\zeta}\).
It follows that each request fails with conditional probability at most
\(e^{-c''\sqrt L}\). A group has only \(L^{O(R)}\) requests, so
its conditional failure probability is bounded by
\(\rho_m=L^{O(R)}e^{-c''\sqrt L}=o(1)\).

If \(F\) is the number of failed groups, iterated conditioning gives
\(\E[F\mid\mathcal F_{\rm match}]\leq M\rho_m\).
Thus \(\Prob(F>M/2\mid\mathcal F_{\rm match})\leq2\rho_m=o(1)\).
No independence between groups is required.

Every selected child is unused, so the selected edges form disjoint
forests. The graph may contain additional edges between their vertices.
For one retained group, let \(v_\ell(r)\) count the depth-\(\ell\)
vertices in community \(r\). Then
\[
  v_0(r)=1,\qquad
  v_{\ell+1}(s)=r_L\sum_{r:s\in\mathcal N(r)}v_\ell(r).
\]
Each target community receives children from \(K-2\) communities.
Induction gives \(v_\ell(r)=d_L^\ell\). Summing over the \(N\)
retained groups proves the stated count. Every nonroot vertex belongs
to a good pool, which gives its margin bound.
\end{proof}

\subsection{The centered forest vector and its quadratic form}

Fix positive weights \(w_0,\ldots,w_R\). Define \(\widetilde y\)
by assigning weight \(d_L^{-\ell/2}w_\ell\) to every retained
depth-\(\ell\) vertex and zero to all other vertices.
Let
\[
  s=\sum_{\ell=0}^R d_L^{\ell/2}w_\ell=O(L^R),\qquad
  \mu=Ns/m.
\]
Subtract \(\mu\) from each coordinate in the active communities to
obtain \(y\), keeping the inactive coordinates at zero. In each
active community,
\begin{equation}
  \sum_i y_i=0,\qquad
  \sum_i y_i^2=N\sum_{\ell=0}^R w_\ell^2-\frac{N^2s^2}{m}
      =N\left(\sum_{\ell=0}^R w_\ell^2-o(1)\right).
  \label{eq:multi-root-energy}
\end{equation}
The number and weight of the coordinates at each depth are the same
in every active community, so the coordinate multisets agree.
Since \(N\leq M\leq m^{1-\zeta}\), the positive coordinates of
\(y\) are precisely the forest vertices for large \(m\).
Their minimum is at least \(c_RL^{-R}\), and
\[
  \max_i y_i^2=O(1),\qquad
  \frac{K\mu}{\min_{y_i>0}y_i}
  \leq\frac Nm L^{O(R)}=o(1).
\]
Together with the uniform degree bound, these estimates verify the
structural assumptions of Lemma~\ref{lem:finite-gram}.

We next estimate \(\mathcal Q(y)\). First, we bound the effect of
subtracting the block means on each of its two terms. Let
\(h=\widetilde y-y=\mu\one_{\rm active}\). The maximum degree
is \(O(L^2)\), so
\begin{align}
  |y^{\mathsf T}Ay-\widetilde y^{\mathsf T}A\widetilde y|
     &\leq2\widetilde y^{\mathsf T}Ah+h^{\mathsf T}Ah
       =O(KN^2L^{2R+2}/m)=o(KNL), \label{eq:forest-centering-error}\\
  \sum_i|D_i^\pi|\,|y_i^2-\widetilde y_i^2|
     &\leq O(L)\,3KN^2s^2/m=o(KNL).
       \label{eq:margin-centering-error}
\end{align}
For the first bound, each term is at most a constant times
\(K(Ns)^2\max_i\deg(i)/m\).

Suppose every root has matched margin at most \(xL\), for a fixed
\(x<a\). There are \(KNd_L^\ell\) selected edges between depths
\(\ell-1\) and \(\ell\).
Every additional edge contributes nonnegatively to the uncentered
quadratic form. The root and nonroot margin bounds therefore give
\begin{equation}
  \frac{\mathcal Q(y)}{KNL}
  \geq -xw_0^2+2\sigma\sum_{\ell=1}^Rw_{\ell-1}w_\ell
            -a\sum_{\ell=1}^Rw_\ell^2-o(1).
  \label{eq:path-variational}
\end{equation}
All error bounds are uniform over \(1\leq N\leq M\).

\subsection{Choosing the weights}

When \(a\geq2\sigma\), set
\begin{equation}
  \tau_0=0,\qquad \tau_{j+1}=\frac{\sigma^2}{a-\tau_j}.
  \label{eq:tau-recurrence}
\end{equation}
These quantities arise from successively eliminating coordinates in the
quadratic form below, and they increase to \(\tau\). At \(a=2\sigma\),
\(\tau_j=\sigma j/(j+1)\).
Let \(T_R\) be the \(R\times R\) tridiagonal matrix with diagonal
\(a\) and adjacent entries \(-\sigma\).
Its least eigenvalue is
\(a-2\sigma\cos(\pi/(R+1))>0\).
Fixing \(w_0=1\), the maximum of the leading expression in
\eqref{eq:path-variational} is
\[
  -x+\sigma^2(T_R^{-1})_{11}=\tau_R-x,
\]
with \((w_1,\ldots,w_R)^{\mathsf T}=\sigma T_R^{-1}e_1\).
Indeed, eliminating the last path coordinate successively gives
\(\sigma^2(T_R^{-1})_{11}=\tau_R\).
The inverse has strictly positive entries, by the nonnegative Neumann
series and connectivity of the path, so all weights are positive.
For a fixed \(x<\tau\), choose a fixed \(R\) with \(\tau_R>x\).
Equations~\eqref{eq:multi-root-energy}--\eqref{eq:path-variational}
then give \(\mathcal Q(y)\geq\delta L\|y\|_2^2\) for fixed \(\delta>0\).
Figure~\ref{fig:forest-response} shows how the forest depth changes
the correction.

\begin{figure}[htbp]
  \centering
  \includegraphics[width=.76\linewidth]{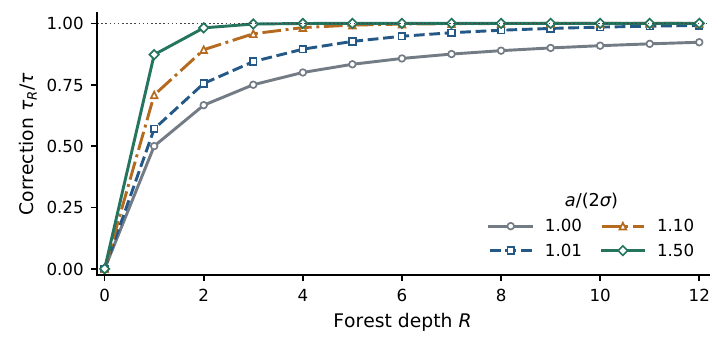}
  \caption{The correction produced by a finite-depth forest.
  Each curve evaluates \eqref{eq:tau-recurrence} at integer depths
  for the indicated value of \(a/(2\sigma)\), with \(\sigma=1\).
  The dotted horizontal line marks the limiting value one.
  At \(a=2\sigma\), the normalized correction is \(R/(R+1)\).
  For any fixed \(x<\tau\), the proof chooses a depth with
  \(\tau_R>x\) before taking the graph limit.}
  \label{fig:forest-response}
\end{figure}

\subsection{Nonoptimality in Theorem~\ref{thm:main}}

\begin{proof}[Proof of the failure assertions]
If \(a\geq2\sigma\) and \(I(\tau)<1\), choose fixed
\(x\in(0,\tau)\) with \(I(x)<1\).
Lemma~\ref{lem:margin-counts} gives a vertex with
\(D_{i,\pi(r)}\leq xL\) in every active block with high probability.
Choose the eligible vertex with the smallest index as the root; if none
exists, choose the vertex with the smallest index in the community.
This selection is \(\mathcal F_{\rm match}\)-measurable.
Apply Lemma~\ref{lem:neutral-forest} with \(M=1\).
Using the weights above, Lemma~\ref{lem:finite-gram} gives a feasible
matrix with strictly larger objective than \(Z^*\).

If \(a<2\sigma\), take \(M=1\) and select the vertex with the smallest
index in each good pool \(G_r\), using the vertex with the smallest
index in the community if the pool is empty. By
\eqref{eq:good-matched-pools}, every pool is nonempty with high
probability. All selected vertices then have matched margins
at most \((a+o(1))L\).
Choose a fixed \(R\) with
\[
  2\sigma\cos\frac{\pi}{R+2}>a,
\]
and \(w_\ell=\sin((\ell+1)\pi/(R+2))>0\).
The analogue of \eqref{eq:path-variational} has leading value
\[
  \left(2\sigma\cos\frac{\pi}{R+2}-a\right)
       \sum_{\ell=0}^Rw_\ell^2>0.
\]
Lemma~\ref{lem:finite-gram} again gives a feasible matrix with a
strictly larger objective than \(Z^*\).
Together with the success proof in Appendix~\ref{sec:success},
this completes Theorem~\ref{thm:main}.
\end{proof}

\subsection{Lower bounds on the objective gain and matrix error}

\begin{proof}[Completion of Theorem~\ref{thm:error}]
The upper bounds were established in Appendix~\ref{app:error-upper}.
Fix \(\varepsilon>0\), write \(r=I(\tau)\in(0,1)\), and choose
fixed \(x\in(0,\tau)\) close enough to \(\tau\) that
\[
  r<I(x)<\min\{1,r+\varepsilon/4\}.
\]
Choose fixed
\(0<\nu<\min\{1-I(x),\varepsilon/4\}\), and set
\(M=\lfloor m^{1-I(x)-\nu}\rfloor\).
The uniform margin-count bound gives at least \(M\) eligible roots in
every active community with high probability. Choose them by a fixed
\(\mathcal F_{\rm match}\)-measurable rule. If a community has fewer
eligible vertices, choose any \(M\) distinct vertices by a fixed rule.

Lemma~\ref{lem:neutral-forest} applies with
\(\zeta=I(x)+\nu>0\) and gives \(N\geq M/2\) successful groups.
Choose \(R\) and the positive path weights so that \(\tau_R>x\).
The energy calculation and Lemma~\ref{lem:finite-gram} give
\begin{equation}
  V_m-\langle A,Z^*\rangle
  \geq\langle A,Z(t)-Z^*\rangle
  \geq cKNL\geq c'kML.
  \label{eq:many-root-gain}
\end{equation}
Here \(K/k\to1\).
Since \(I(x)+\nu<r+\varepsilon/2\), dividing by \(nL=kmL\)
gives
\(\mathcal G_m\geq m^{-r-\varepsilon}\) for large \(m\).

The gain in \eqref{eq:many-root-gain} diverges.
Every feasible \(Z\) with objective at least \(V_m-1\) therefore has
\(\langle A,Z-Z^*\rangle\geq c'kML/2\) for large \(m\).
The bound \eqref{eq:gain-forces-error} holds simultaneously for
all feasible matrices, so it implies
\(\mathcal E_m(Z)\geq c''M/m\geq m^{-r-\varepsilon}\).
The two bounds hold simultaneously for all such matrices on an event
whose probability tends to one.
\end{proof}

%% file: appendices/bulk.tex
\section{Macroscopic optimal value and overlap}
\label{sec:bulk}

We prove Proposition~\ref{prop:bulk-value} by matching an upper bound
valid for every feasible matrix with a feasible construction. For the
upper bound, we separate the planted subspace from its orthogonal
complement. For the lower bound, we construct a polynomial Gram matrix
and adjust it to satisfy the unit-diagonal, row-sum, and nonnegativity
constraints. Local constructions for SDP values also appear in
\citet{MontanariSen16} and \citet{FanMontanari17}. Here we must additionally
enforce the row sums and entrywise nonnegativity.

\subsection{An objective bound uniform over the feasible set}

Let \(u_i=\max_s|d_{i,s}-\E d_{i,s}|\).
Bernstein's inequality, a union bound over \(k\) communities, and
integration of the resulting tail give
\[
  \E u_i\leq C\{\sqrt{L\log k}+\log k\}=o(L).
\]
The tail integration is the same as in \eqref{eq:max-degree-mean}.
Markov's inequality then gives
\(\sum_i u_i=o_{\Prob}(nL)\).
For every feasible \(Z\), entrywise nonnegativity and row sums \(m\)
imply
\begin{equation}
  |\langle QW,Z\rangle|,\quad|\langle WQ,Z\rangle|,\quad
  |\langle QWQ,Z\rangle|\leq\sum_i u_i.
  \label{eq:community-noise-cross}
\end{equation}
Indeed, \(|(QW)_{ij}|\leq u_j/m\). Since each column of \(Z\)
is nonnegative and sums to \(m\), summing over its entries gives
the first bound.
Symmetry gives the second. A further block average bounds
\(|(QWQ)_{ij}|\) by
\(m^{-2}\sum_{v\in C_{c(j)}}u_v\), yielding the third.

Since \(W=\Pi W\Pi+QW+WQ-QWQ\),
\(\Pi Z\Pi\succeq0\), and
\(\Tr(\Pi Z)=n(1-\theta_m(Z))\), the norm bound yields
\[
  \langle W,Z\rangle
  \leq(2\sigma+o_{\Prob}(1))nL(1-\theta_m(Z))
       +o_{\Prob}(nL)
\]
uniformly in \(Z\).
Also
\(\E A=(p-q)Z^*+qJ-pI\), where \(J=\one\one^{\mathsf T}\).
Consequently,
\begin{equation}
  \frac{\langle A,Z\rangle}{nL}
  \leq\beta+a\theta_m(Z)+2\sigma(1-\theta_m(Z))
       +o_{\Prob}(1),
  \label{eq:macro-upper}
\end{equation}
with an error term uniform over \(Z\). The zero diagonal contributes
\(-\alpha/m=o(1)\). Since \(0\leq\theta_m(Z)\leq1\),
this proves the required upper bound for \(V_m\).

\subsection{Fixed-degree polynomial moments}

Set \(Y=A/L\). For every fixed radius \(r\), the expected number
of vertices whose radius-\(r\) ball contains a cycle is
\(L^{O(r)}\). This follows from the connected unicyclic configurations
counted in the marked-resolvent proof, without requiring any vertex
to be marked.
With high probability there are at most \(n^{1/4}\) such vertices.
At all other vertices, the uniform total-degree bound gives the same
closed-walk counts as \eqref{eq:diagonal-walk-count}.
Thus, for each fixed \(j\),
\[
  (Y^j)_{ii}=\mu_j+o_{\Prob}(1)
\]
uniformly outside the exceptional set, where
\(\mu_{2r}=\frac1{r+1}\binom{2r}{r}\sigma^{2r}\) and
\(\mu_{2r+1}=0\). The norm bound
\(\|Y\|\leq\max_i\deg(i)/L=O_{\Prob}(L)\)
shows that these exceptional vertices contribute a vanishing amount
to the normalized traces.
Hence every fixed polynomial \(g\) satisfies
\begin{equation}
  \frac1n\Tr g(Y)\xrightarrow{\Prob}
     \int g(\lambda)\,\mu_{\rm sc,\sigma}(d\lambda),
  \label{eq:unmarked-moments}
\end{equation}
where \(\mu_{\rm sc,\sigma}\) is the centered semicircle law of
variance \(\sigma^2\), with density
\(\sqrt{4\sigma^2-\lambda^2}/(2\pi\sigma^2)\)
on \([-2\sigma,2\sigma]\).
The radii and polynomial degrees remain fixed as \(m\to\infty\).

\subsection{Exact constraints for the polynomial Gram matrix}

We use a power of a shifted adjacency matrix to give greater weight
to eigenvalues near the upper spectral edge. Fix an integer \(R\geq1\)
and set
\[
  F=(Y+2\sigma I)^R,\qquad
  D_{ii}=((F^2)_{ii})^{-1/2},\qquad C=DF^2D.
\]
The polynomial has nonnegative coefficients, so \(F\geq0\) entrywise,
and \(F_{ii}\geq(2\sigma)^R>0\).
Thus \(C\succeq0\), \(C\geq0\), and \(\diag C=\one\).
Its maximum row sum is
\(q_R=O_R(L^{2R})\) on the uniform degree event, since
\(D_{ii}\leq(2\sigma)^{-R}\).

To impose the row constraint, first project out the all-ones direction.
Let \(P_0=I-J/n\) and \(H=P_0CP_0\).
Then \(H\succeq0\), \(H\one=0\), and, writing
\(s_i=\sum_j C_{ij}\),
\[
  H_{ij}=C_{ij}-\frac{s_i+s_j}{n}
          +\frac{\sum_\ell s_\ell}{n^2}.
\]
Therefore
\(H_{ij}\geq-2q_R/n\) and \(H_{ii}=1+O(q_R/n)\) uniformly.
Adding \(J/k\) will give row sums \(m\). To recover unit diagonal
while preserving positive semidefiniteness, first shrink \(H\) by
a fixed factor \(1-\zeta\), where \(\zeta\in(0,1)\), and define
\begin{equation}
  \delta_i=1-\frac1k-(1-\zeta)H_{ii},\qquad
  \bar\delta=\frac1n\sum_i\delta_i,\qquad
  r_i=\frac{\delta_i-\bar\delta/(n-1)}{1-2/n}.
  \label{eq:bulk-diagonal-repair}
\end{equation}
For large \(m\), \(\delta_i=\zeta+o(1)\) uniformly, so
\(r_i>0\) and \(r_i=O(\zeta)\).
Summing \eqref{eq:bulk-diagonal-repair} gives
\(\sum_i r_i=n^2\bar\delta/(n-1)\), and hence
\[
  \bigl(P_0\diag(r)P_0\bigr)_{ii}
  =(1-2/n)r_i+\frac{\sum_jr_j}{n^2}=\delta_i.
\]
It follows that
\begin{equation}
  Z_{R,\zeta}=\frac Jk+(1-\zeta)H+P_0\diag(r)P_0
  \label{eq:bulk-feasible-gram}
\end{equation}
is positive semidefinite, has unit diagonal, and has row sums \(n/k=m\).
Its off-diagonal entries are at least
\(1/k-O(q_R/n)-O(\zeta/n)>0\), because
\(kq_R/n\to0\). Thus \(Z_{R,\zeta}\in\mathcal F\).

\subsection{The objective value and overlap}

Let
\[
  c_R=\int(\lambda+2\sigma)^{2R}\,
                 \mu_{\rm sc,\sigma}(d\lambda)>0.
\]
Outside the exceptional set defined above,
\((F^2)_{ii}=c_R+o_{\Prob}(1)\) uniformly.
We first show that the diagonal normalization has only a small effect:
\begin{equation}
  \|C-c_R^{-1}F^2\|_*=o_{\Prob}(n),
  \label{eq:polynomial-normalization}
\end{equation}
where \(\|\cdot\|_*\) denotes the nuclear norm.
For the Gram factors, we have
\[
  \frac1n\|(D-c_R^{-1/2}I)F\|_F^2=o_{\Prob}(1).
\]
Outside the exceptional set, \(D_{ii}\to c_R^{-1/2}\) uniformly.
The contribution from these rows tends to zero because this convergence
is uniform and \(n^{-1}\Tr F^2\to c_R\) by
\eqref{eq:unmarked-moments}. On the exceptional rows, use
\((D_{ii}-c_R^{-1/2})^2(F^2)_{ii}
\leq2+2c_R^{-1}(F^2)_{ii}\)
and the polynomial bound \(\|F\|=O_R(L^R)\).
There are at most \(n^{1/4}\) such rows, so their contribution after
division by \(n\) also tends to zero. Factoring
\(C=(DF)(DF)^{\mathsf T}\), we can now apply
\[
  \|UU^{\mathsf T}-VV^{\mathsf T}\|_*
  \leq\|U-V\|_F(\|U\|_F+\|V\|_F),
\]
with \(\|DF\|_F^2=n\) and \(\|F\|_F^2/n\to c_R\).
This proves \eqref{eq:polynomial-normalization}.

The difference \(H-C\) has rank at most two and norm \(O(q_R)\),
so its nuclear norm is \(o(n)\).
Also \(\E A/L=UMU^{\mathsf T}-\alpha I/m\), where \(U\)
contains normalized block indicators, \(M\) has diagonal \(\alpha\)
and off-diagonal \(\beta\), and \(\|M\|=O(L)\).
The contribution of the rank-\(k\) term to the normalized inner product
with \(F^2\) is at most
\(k\|M\|\|F\|^2/n=o(1)\); the diagonal term is negligible as well.
Equations~\eqref{eq:unmarked-moments}--\eqref{eq:polynomial-normalization}
and \(\|W/L\|=O_{\Prob}(1)\) therefore imply
\begin{equation}
  \frac{\langle W,H\rangle}{nL}
  \xrightarrow{\Prob}
  \lambda_R=
  \frac{\int\lambda(\lambda+2\sigma)^{2R}\,
                      \mu_{\rm sc,\sigma}(d\lambda)}
       {\int(\lambda+2\sigma)^{2R}\,
                      \mu_{\rm sc,\sigma}(d\lambda)}
  =\frac{4\sigma R}{2R+3}.
  \label{eq:bulk-polynomial-value}
\end{equation}
For the last equality, substitute \(\lambda=4\sigma u-2\sigma\).
The weighted density is proportional to
\(u^{2R+1/2}(1-u)^{1/2}\) on \([0,1]\), whose mean is
\((2R+3/2)/(2R+3)\).
In particular, \(\lambda_R\uparrow2\sigma\).

Figure~\ref{fig:bulk-weights} illustrates the effect of the polynomial
weight. As \(R\) increases, the normalized spectral density places
more mass near the upper edge. Its mean is \(\lambda_R\), which is
the limiting noise value in \eqref{eq:bulk-polynomial-value}.

\begin{figure}[htbp]
  \centering
  \includegraphics[width=\linewidth]{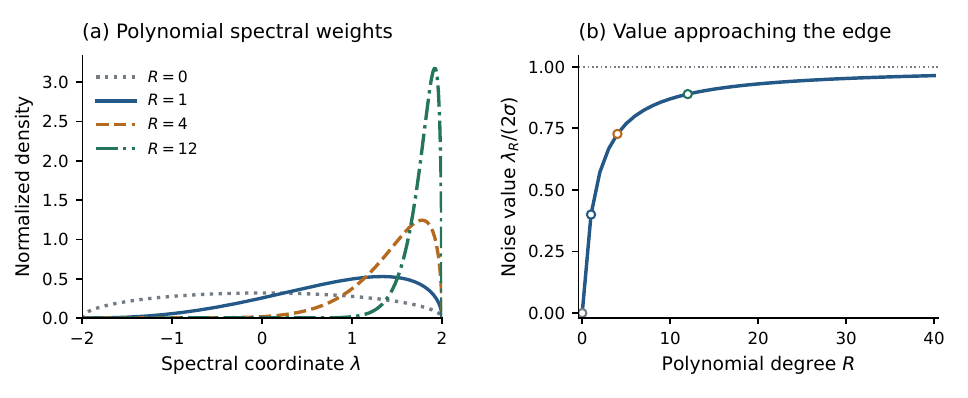}
  \caption{Polynomial weighting toward the upper spectral edge.
  (a) Normalized densities proportional to
  \((\lambda+2\sigma)^{2R}\sqrt{4\sigma^2-\lambda^2}\)
  on \([-2\sigma,2\sigma]\), with \(\sigma=1\).
  The unweighted semicircle density is included as \(R=0\).
  (b) Their normalized means,
  \(\lambda_R/(2\sigma)=2R/(2R+3)\), with circles at the degrees
  shown in panel~(a). The dotted line marks the spectral edge.
  For each fixed \(R\geq1\), this is the limiting noise value of
  the centered Gram matrix \(H\). The diagonal adjustment is
  accounted for separately in the proof.}
  \label{fig:bulk-weights}
\end{figure}

The positive semidefinite diagonal correction has trace \(O(\zeta n)\).
Together with the norm bound on \(W\), this bounds its contribution
to the noise objective below by \(-C\zeta nL\).
The contribution of \(J/k\) is \(o_{\Prob}(nL)\), since its trace
is \(n/k\). Finally, \(\alpha>\beta\) and feasibility give
\(\langle\E A,Z_{R,\zeta}\rangle\geq\beta nL-o(nL)\).
Thus
\[
  \frac{V_m}{nL}
  \geq\beta+(1-\zeta)\lambda_R-C\zeta-o_{\Prob}(1).
\]
For any fixed accuracy, choose \(R\) sufficiently large and
\(\zeta\) sufficiently small, and then let \(m\to\infty\) with
both fixed.
The planted feasible value satisfies
\(\langle A,Z^*\rangle/(nL)\to\alpha=\beta+a\);
this follows also from the bound on \(\sum_i u_i\).
Together with \eqref{eq:macro-upper}, these lower bounds prove
\eqref{eq:bulk-value}, including \(a=2\sigma\).

To obtain the overlap limits, consider any matrix with objective at
least \(V_m-s_m\), where \(s_m=o(nL)\). If \(a>2\sigma\),
\eqref{eq:macro-upper} and the limiting optimal value bound its
normalized objective deficit below by
\((a-2\sigma)(1-\theta_m(Z))-o_{\Prob}(1)\).
The deficit tends to zero, so \(\theta_m(Z)\to1\) uniformly over
these matrices. If \(a<2\sigma\), the corresponding lower bound is
\((2\sigma-a)\theta_m(Z)-o_{\Prob}(1)\), which gives
\(\theta_m(Z)\to0\) with the same uniformity.
This proves Proposition~\ref{prop:bulk-value}.

%% file: appendices/refinement.tex
\section{Row clustering and same-graph refinement}
\label{sec:refinement}

We first show that row clustering produces an initial labeling with
\(o(m)\) errors. We then prove that refinement reduces the error count
on an event that holds for every sufficiently small vertex set. Because
this event also applies to sets selected from the graph, the same graph
can be used for initialization and every refinement step. The procedure
requires only \(A\) and the known community size \(m\).

\subsection{A deterministic row-clustering guarantee}

For an input matrix \(Z\), compute all \(\ell_1\) distances
between its rows. Initialize the remaining set to \([n]\).
For \(r=1,\ldots,k\), choose a remaining row whose radius-\(m/2\)
ball contains the largest number of remaining rows, retain its index,
and remove the remaining rows within distance \(m\) of it.
Break ties by smallest vertex index. After \(k\) stages, assign each
of the original rows to its nearest retained center, with ties broken
by center order. If the remaining set becomes empty before stage \(k\),
output arbitrary labels. Under the condition in the next lemma, all
\(k\) stages are completed. The distances and ball counts can be
computed in \(O(n^3)\) operations.

Clustering SDP rows is analyzed by
\citet[Section~2.3]{FeiChen19}, and ball-based rounding appears in
\citet[Algorithm~2]{GaoMaZhangZhou17}. The following deterministic
bound applies to the version just specified.

\begin{lemma}[Row rounding]
\label{lem:row-rounding}
If \(D=\|Z-Z^*\|_{1,\mathrm{entry}}\) satisfies
\(8D/m<m/2\), the procedure performs all \(k\) stages and,
after a permutation of labels, makes at most \(8D/m\) errors.
\end{lemma}

\begin{proof}
Let \(z_i\) denote row \(i\) of \(Z\), and let
\(p_r=\one_{C_r}\) be the common row of \(Z^*\) for vertices in
community \(r\). Call row \(i\in C_r\) good if
\(\|z_i-p_r\|_1\leq m/8\).
There are at most \(B=8D/m\) bad rows.
Good rows in one block have pairwise distance at most \(m/4\);
good rows in distinct blocks have distance at least \(7m/4\).

Suppose the earlier stages have each removed the good rows of one
block and no good rows from other blocks. Every unprocessed block
then has at least \(m-B>B\) remaining good rows. This is also true
before the first stage. The radius-\(m/2\) ball of any one of them
contains all good rows of that block.
The selected ball therefore contains a good row. It cannot contain good
rows from two communities, because its diameter is at most \(m\).
Its center is within \(m/2\) of a good row of one block.
The radius-\(m\) removal therefore includes all good rows of that
block and none from any other block: the distances are respectively
at most \(3m/4\) and at least \(5m/4\).
Induction proves there are \(k\) centers, one associated with each block.

Each center is within \(5m/8\) of the true row for its associated community.
Every good row is therefore within \(3m/4\) of its own center
and at least \(5m/4\) from the other centers.
The final nearest-center assignment labels all good rows correctly.
\end{proof}

By \eqref{eq:coarse-initialization}, when \(a>2\sigma\), every feasible
input with objective at least \(V_m-1\) gives
\[
  h_0\leq16n\,\mathcal E_m(Z)\leq16n\,m^{-c_{\mathrm{init}}}=o(m)
\]
initial errors, simultaneously for all such inputs, with high probability.
Here \(D=2nm\,\mathcal E_m(Z)\), and \(k=O(\log m)\) makes the
last bound \(o(m)\).

\subsection{Uniform sparsity of small vertex sets}

Let \(e(U)\) count undirected edges with both endpoints in \(U\).

\begin{lemma}[Uniform small-set sparsity]
\label{lem:small-set}
For each fixed \(c>0\), there is a fixed \(\eta>0\) such that,
with probability tending to one,
\[
  e(U)<cL|U|\quad\text{for every }1\leq|U|\leq\eta m.
\]
\end{lemma}

\begin{proof}
Write \(u=|U|\) and \(t_u=\lceil cLu\rceil\).
If \(t_u>\binom u2\), no set of size \(u\) violates the claimed
bound. Otherwise, a union bound over the vertex sets and the required
edges gives
\[
  \Prob\{\exists U:|U|=u,\ e(U)\geq cLu\}
  \leq\binom nu\binom{\binom u2}{t_u}p^{t_u}
  \leq
  \left(\frac{en}{u}\right)^u
  \left(\frac{e\alpha u}{2cm}\right)^{cLu},
\]
provided the base of the second factor is less than one.
Let \(C=e\alpha/(2c)\), and choose \(\eta\) so small that
\(v=\log(m/u)\geq2\max\{\log C,0\}+2\) for \(u\leq\eta m\).
The logarithm of the last expression is
\[
  u\{1+\log k+v-cL(v-\log C)\}
  \leq-\frac{c}{4}Luv
\]
for all large \(m\), uniformly in \(u\), because \(\log k=O(\log L)\)
and \(v-\log C\geq v/2\).
Since \(v\geq2\), summing this bound over the possible values of
\(u\) gives \(o(1)\).
\end{proof}

\subsection{Synchronous refinement}

Given labels \(\widehat c^{(t)}\), assign each vertex at the next step
to a label that maximizes
\[
  s\longmapsto\sum_{j:\widehat c^{(t)}(j)=s}A_{ij},
\]
using the same current labels for all vertices and breaking ties by the
smallest label. Perform
\(T=\lceil\log_2n\rceil+2\) updates.
Fix the permutation aligning the initialization to the true labels;
the analysis uses this same alignment at every iteration.

For fixed \(d\in(0,a)\), put
\[
  \mathcal B_d=\{i:\Delta_i<dL\},\qquad B_d=|\mathcal B_d|.
\]
If the current error set is \(E\), then each difference between two
neighbor counts differs from its value under the true labels by at most
\(2\deg_E(i)\). Thus any vertex outside \(\mathcal B_d\) that is
incorrect after the update must have
\(\deg_E(i)\geq dL/2\).

We now use the event in Lemma~\ref{lem:small-set}, with \(c<d/12\).
If \(h=|E|\geq1\) and at least \(\lceil h/2\rceil\) vertices
outside \(\mathcal B_d\) are incorrect after the update, let \(F\)
consist of exactly that many of these vertices. For \(h\leq\eta m/2\),
the set \(U=E\cup F\)
satisfies \(|U|\leq\eta m\) and \(|U|\leq3|F|\). Therefore
\[
  \frac{dL}{2}|F|\leq\sum_{i\in F}\deg_E(i)
  \leq2e(U)<6cL|F|,
\]
a contradiction. If \(h=0\), all vertices outside
\(\mathcal B_d\) retain their correct labels.
In both cases,
\begin{equation}
  h_{t+1}\leq\lfloor h_t/2\rfloor+B_d.
  \label{eq:refinement-contraction}
\end{equation}
In particular, if \(h_0=o(m)\) and \(B_d=o(m)\), every iterate
stays below \(\max\{h_0,2B_d\}=o(m)\).
The small-set bound holds simultaneously for every possible error set,
so it applies at each iteration even though the labels depend on the
graph. Iterating the recurrence gives \(h_T\leq h_0/2^T+2B_d\).
Since \(h_0\leq n\) and \(h_T\) is an integer, \(h_T\leq2B_d\).

To illustrate this bound, define \(H_0=h_0\) and
\(H_{t+1}=\lfloor H_t/2\rfloor+B_d\).
On the same small-set event, induction gives \(h_t\leq H_t\).
Figure~\ref{fig:refinement-bound} plots this deterministic comparison
for a few fixed values of \(B_d\). When \(B_d=0\), the bound
reaches zero; when \(B_d>0\), the guarantee allows up to \(2B_d\)
errors after the stated number of rounds.

\begin{figure}[htbp]
  \centering
  \includegraphics[width=.78\linewidth]{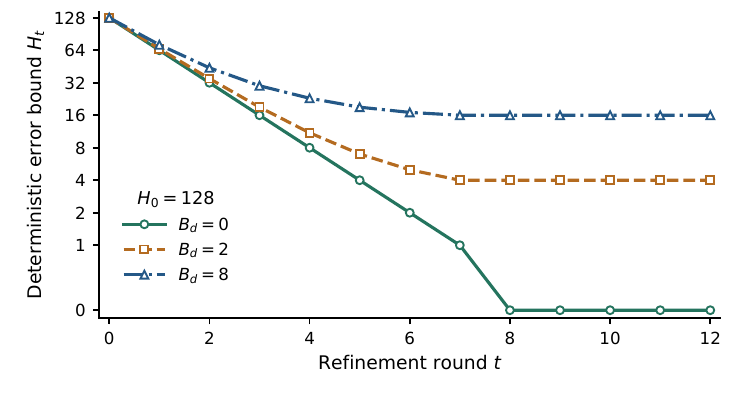}
  \caption{The deterministic refinement bound from
  \eqref{eq:refinement-contraction}, initialized at \(H_0=128\).
  Each curve iterates \(H_{t+1}=\lfloor H_t/2\rfloor+B_d\) for
  the indicated number of low-margin vertices. The \(B_d=0\) curve
  reaches zero after eight rounds; the other two curves settle at
  \(2B_d\). The vertical axis is logarithmic above one and linear
  near zero. These sequences illustrate the proof's upper bound on
  errors under its small-set hypotheses.}
  \label{fig:refinement-bound}
\end{figure}

\begin{proof}[Proof of Corollary~\ref{cor:postprocess}]
When \(I(0)>1\), choose fixed \(d>0\) small enough that \(I(d)>1\).
A union bound over \(n(k-1)\) comparisons gives
\(\Prob(B_d>0)\leq m^{1-I(d)+o(1)}=o(1)\).
The coarse matrix-error bound and Lemma~\ref{lem:row-rounding}
give \(h_0=o(m)\) simultaneously for every feasible input with
objective at least \(V_m-1\). When these bounds and the small-set
sparsity bound hold, \eqref{eq:refinement-contraction} reduces the
error count to zero. All true margins are positive, so subsequent
updates preserve the correct labels. These events hold jointly with
probability tending to one.
\end{proof}

\begin{proposition}[A high-probability label-error bound]
\label{prop:label-rate}
Assume \(a>2\sigma\).
For every fixed \(\varepsilon>0\), the same procedure satisfies
\[
  \min_{\pi\in S_k}\frac1n
  \#\{i:\widehat c(i)\ne\pi(c(i))\}
  \leq m^{-I_{\alpha,\beta}(0)+\varepsilon}
\]
with high probability, simultaneously for every feasible input with
objective at least \(V_m-1\).
\end{proposition}

\begin{proof}
It suffices to prove the bound for \(0<\varepsilon<I(0)/2\).
Choose fixed \(d>0\) such that \(I(d)>I(0)-\varepsilon/4\).
The margin tail gives
\(\E B_d\leq n m^{-I(d)+o(1)}\).
Markov's inequality implies
\[
  B_d\leq n m^{-I(0)+\varepsilon/2}
\]
with high probability.
This is \(o(m)\), since \(k=O(L)\) and \(I(0)-\varepsilon/2>0\).
The initial error count is also \(o(m)\), so the recurrence
gives \(h_T\leq2B_d\leq n m^{-I(0)+\varepsilon}\) for large \(m\).
\end{proof}

The proposition gives a high-probability upper bound on label error
with exponent \(I(0)\). It does not provide a matching lower bound
or an expected-risk guarantee. For the SDP matrix itself,
Theorem~\ref{thm:error} gives matching upper and lower bounds with
exponent \(I(\tau)\).

%% file: appendices/consequences.tex
\section{Information benchmark and threshold consequences}
\label{app:consequences}

For fixed edge-probability parameters, the tightness theorem gives a
critical value of the community ratio \(k/\log m\). We derive this
ratio and compare the resulting boundary with the information-theoretic
threshold. We then describe its limit as the background parameter grows.

\subsection{Extreme margins and the critical community ratio}

The rate function is continuous and strictly decreasing on \((-\infty,a]\), with
\begin{equation}
  I_{\alpha,\beta}(0)=(\sqrt\alpha-\sqrt\beta)^2,
  \qquad I_{\alpha,\beta}(a)=0.
  \label{eq:rate-endpoints}
\end{equation}
Its derivative is
\[
  I'_{\alpha,\beta}(x)
  =\log\frac{x+\sqrt{x^2+4\alpha\beta}}{2\alpha}<0
  \qquad(x<a).
\] 

When \(I_{\alpha,\beta}(0)>1\), the smallest margin divided by
\(L\) converges to a positive constant. Define \(x_*\in(0,a)\) by
\begin{equation}
  I_{\alpha,\beta}(x_*)=1.
  \label{eq:extreme-margin}
\end{equation}
The degree-margin tail and count bounds give
\begin{equation}
  L^{-1}\min_i\Delta_i\xrightarrow{\Prob}x_*.
  \label{eq:minimum-margin-limit}
\end{equation}
To prove this, fix \(0<\varepsilon<a-x_*\).
A union bound over \(n(k-1)=mL^{O(1)}\) comparisons and
Lemma~\ref{lem:margin-ldp} exclude margins below
\((x_*-\varepsilon)L\) with high probability.
Lemma~\ref{lem:margin-counts} also shows that, with high probability,
every ordered pair of distinct communities has vertices whose margin
is at most \((x_*+\varepsilon)L\). Together these bounds prove
\eqref{eq:minimum-margin-limit}.

Since \(I\) is decreasing, the local success condition is
\(x_*>\tau\). Combining it with the bulk condition gives the
following expression for the critical ratio.

\begin{corollary}[Critical community ratio]
\label{cor:capacity}
If \(I_{\alpha,\beta}(0)>1\), define
\begin{equation}
  \gamma_{\mathrm{SDP}}(\alpha,\beta)=
  \begin{cases}
    x_*(a-x_*)/\beta,&x_*<a/2,\\
    a^2/(4\beta),&x_*\geq a/2.
  \end{cases}
  \label{eq:critical-ratio}
\end{equation}
Set \(\gamma_{\mathrm{SDP}}=0\) when \(I_{\alpha,\beta}(0)\leq1\).
For fixed \(0<\gamma<\gamma_{\mathrm{SDP}}\), \(Z^*\) is the unique
SDP optimizer with high probability. For fixed
\(\gamma>\gamma_{\mathrm{SDP}}\), it is not optimal with high probability.
\end{corollary}

In the first branch, the correction reaches the smallest margin before
the bulk spectral gap closes. In the second branch, the bulk gap closes
first. The branches meet at \(x_*=a/2\).

\begin{proof}[Proof of Corollary~\ref{cor:capacity}]
If \(I(0)\leq1\), then \(I(\tau)<1\) whenever
\(a\geq2\sigma\), since \(\tau>0\); the remaining case is the
bulk failure condition. Assume now that \(I(0)>1\).
For \(0<\gamma<a^2/(4\beta)\), the map
\(\gamma\mapsto\tau\) increases from zero to \(a/2\), and
\(I(\tau)>1\) is equivalent to \(\tau<x_*\).
If \(x_*<a/2\), equality occurs at
\(\beta\gamma=x_*(a-x_*)\).  If \(x_*\geq a/2\), the
bulk edge is reached first.  Theorem~\ref{thm:main} gives the stated
conclusions. When \(x_*<a/2\), the case \(a=2\sigma\) satisfies
\(I(\sigma)<1\), so the theorem also gives nonoptimality there.
\end{proof}

\subsection{The information threshold and the SDP gap}

The achievability part of the information-theoretic benchmark follows
from the growing-community minimax bound of \citet{ZhangZhou16}.
We give the parameter conversion below. For the converse, we show that
swapping two vertices between communities increases the likelihood.

\begin{proposition}[Information-theoretic benchmark]
\label{prop:information}
Exact recovery of the partition up to community relabeling is possible
with high probability if
\begin{equation}
  I_{\alpha,\beta}(0)=(\sqrt\alpha-\sqrt\beta)^2>1.
  \label{eq:information-threshold}
\end{equation}
If \(I_{\alpha,\beta}(0)<1\) and the partition is uniform over
balanced partitions, every estimator has exact-recovery probability
tending to zero.
\end{proposition}

\begin{proof}[Proof of Proposition~\ref{prop:information}]
For achievability, apply \citet[Theorem~1.1]{ZhangZhou16}.
Their community count \(K\) is our \(k\), their balance parameter
is one, and their edge-probability parameters \(a_{\mathrm{ZZ}},b_{\mathrm{ZZ}}\)
are \(np,nq\), respectively.
Write \(\mathcal D_{1/2}\) for the Bernoulli R\'enyi divergence used
there, to distinguish it from our margin rate function.  Then
\[
  \mathcal D_{1/2}
  =-2\log\!\left(\sqrt{pq}+\sqrt{(1-p)(1-q)}\right),
  \qquad
  m\mathcal D_{1/2}=I_{\alpha,\beta}(0)L+o(1).
\]
Their assumption \(n\mathcal D_{1/2}/(k\log k)\to\infty\)
holds because \(L/\log k\to\infty\).
Their upper bound gives an estimator whose expected misclassification
fraction, minimized over label permutations, is at most
\(m^{-I(0)+o(1)}\), uniformly over balanced partitions.
Any incorrectly recovered partition has loss at least \(1/n\). Thus Markov's
inequality bounds its probability by
\(n m^{-I(0)+o(1)}=o(1)\) when \(I(0)>1\).

For the converse, we use the local-swap argument of
\citet[Section~5]{AbbeBandeiraHall16} inside two fixed planted
communities.  Choose \(\varepsilon>0\) with \(I(-\varepsilon)<1\).
Lemma~\ref{lem:margin-counts} implies that, with high probability,
both sets
\[
  \{i\in C_1:D_{i,2}\leq-\varepsilon L\},\qquad
  \{j\in C_2:D_{j,1}\leq-\varepsilon L\}
\]
are nonempty.  Choose one vertex from each.
Swapping their communities increases the number of within-community
edges, counted once, by
\[
  -D_{i,2}-D_{j,1}-2A_{ij}\geq2\varepsilon L-2>0.
\]
The swap preserves the community sizes, so it does not change the number
of possible within-community pairs. The log-likelihood gain equals
this edge-count gain multiplied by
\[
  \log\frac{p(1-q)}{q(1-p)}>0.
\]
Edges to all other communities retain their likelihood factors.
Thus the planted partition is not a maximum likelihood partition
with high probability.  Under the uniform balanced prior, maximum
likelihood is also maximum a posteriori estimation and maximizes the
probability of exact recovery.  This proves the converse.
\end{proof}

For \(a>2\sigma\), the strict decrease of \(I\) gives
\begin{equation}
  I_{\alpha,\beta}(\tau)<I_{\alpha,\beta}(0).
  \label{eq:strict-gap}
\end{equation}
The next corollary shows that every fixed positive community ratio
admits a gap between information-theoretic recovery and SDP tightness.

\begin{corollary}[A recovery gap at every positive ratio]
\label{cor:gap}
For every fixed \(\beta,\gamma>0\), there is
\(\varepsilon=\varepsilon(\beta,\gamma)>0\) such that every fixed
\(\alpha\) satisfying
\[
  1<\sqrt\alpha-\sqrt\beta<1+\varepsilon
\]
permits information-theoretic exact recovery with high probability,
while \(Z^*\) is not optimal for \eqref{eq:sdp} with high probability.
\end{corollary}

\begin{proof}[Proof of Corollary~\ref{cor:gap}]
Fix \(\beta\), write \(\alpha_\delta=(\sqrt\beta+\delta)^2\),
and let \(\delta\downarrow1\).
For \(\delta>1\), \(I_{\alpha_\delta,\beta}(0)=\delta^2>1\),
and the root \(x_*(\delta)>0\) of
\(I_{\alpha_\delta,\beta}(x)=1\) tends to zero.
To see this, fix any \(u\in(0,2\sqrt\beta+1)\).
At \(\delta=1\), strict monotonicity gives
\(I_{\alpha_1,\beta}(u)<1\).
By continuity, this inequality still holds for \(\delta>1\) sufficiently
close to one. Since the rate is decreasing, \(0<x_*(\delta)<u\).

On the other hand,
\(a_\delta=\alpha_\delta-\beta\to2\sqrt\beta+1>0\).
Thus \(x_*(\delta)<a_\delta/2\) near one, and
Corollary~\ref{cor:capacity} gives
\[
  \gamma_{\mathrm{SDP}}(\alpha_\delta,\beta)
  =\frac{x_*(\delta)(a_\delta-x_*(\delta))}{\beta}
  \longrightarrow0.
\]
For any fixed \(\gamma>0\), this critical ratio is less than
\(\gamma\) throughout a sufficiently short interval
\(1<\delta<1+\varepsilon\).
Proposition~\ref{prop:information} gives information-theoretic recovery,
and Corollary~\ref{cor:capacity} gives nonoptimality of \(Z^*\).
\end{proof}

The critical-ratio formula is continuous where \(I_{\alpha,\beta}(0)>1\).
For fixed \(\gamma\), the strict gap conditions therefore continue
to hold in an open neighborhood of each such pair \((\alpha,\beta)\).

\subsection{The large-background limit}

In the coordinate \(\delta=\sqrt\alpha-\sqrt\beta\), the boundary
simplifies as \(\beta\to\infty\). Define
\begin{equation}
  f(\gamma)=\begin{cases}
    1+\gamma/4,&0<\gamma\leq4,\\
    \sqrt\gamma,&\gamma\geq4.
  \end{cases}
  \label{eq:limiting-sdp-boundary}
\end{equation} 

\begin{proposition}[Sequential-limit phase diagram]
\label{prop:sequential-limit}
Fix \(\gamma,\delta>0\) with \(\delta\ne f(\gamma)\).
There exists \(\beta_0(\gamma,\delta)\) such that for every fixed
\(\beta>\beta_0\), with \(\alpha=(\sqrt\beta+\delta)^2\),
as \(m\to\infty\) and \(k/L\to\gamma\), \(Z^*\) is the unique
SDP optimizer with high probability if \(\delta>f(\gamma)\), and
is not optimal with high probability if \(\delta<f(\gamma)\).
\end{proposition}

The proposition first takes \(m\to\infty\) at fixed
\(\alpha,\beta,\gamma\), and then takes \(\beta\to\infty\).
It does not cover \(\beta=\beta_m\to\infty\).
Figure~\ref{fig:sequential-phase} compares this limiting curve with
the fixed-\(\beta=1\) boundary. The junction \((\gamma,\delta)=(4,2)\)
is the junction in the sequential limit; for finite \(\beta\), the
junction can be different.

\begin{figure}[t]
\centering
\includegraphics[width=\linewidth]{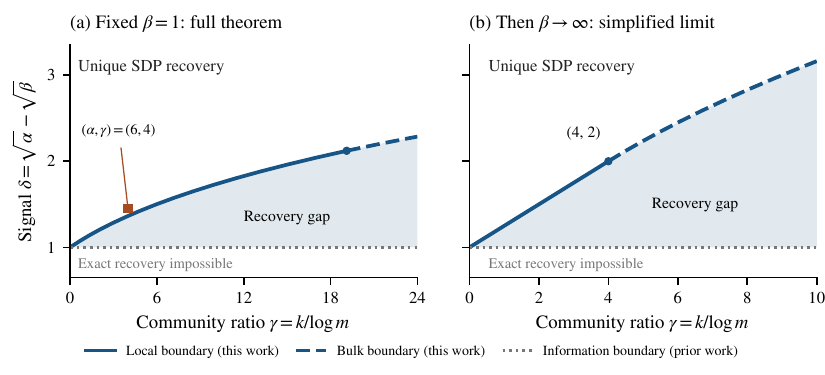}
\caption{The fixed-background and large-background boundaries.
The left panel uses \(\beta=1\); the right panel takes the graph
limit first and then \(\beta\to\infty\). The two curves have
different local--bulk junctions. Recovery claims concern the strict
regions on either side of the curves. The square marks
\((\alpha,\beta,\gamma)=(6,1,4)\), where the SDP is tight although
\(\sqrt6-1<f(4)=2\).}
\label{fig:sequential-phase}
\end{figure}

\begin{proof}[Proof of Proposition~\ref{prop:sequential-limit}]
For fixed \(\delta>0\) and uniformly over \(u\) in compact subsets
of \((-\infty,2\delta)\), the explicit optimizer in
Lemma~\ref{lem:margin-ldp} satisfies
\[
  \theta_{\sqrt\beta u}
  =\frac{u-2\delta}{2\sqrt\beta}+O(\beta^{-1}).
\]
For bounded \(v\), the cumulant expression has the expansion
\[
  \alpha(e^{v/\sqrt\beta}-1)
  +\beta(e^{-v/\sqrt\beta}-1)
  =2\delta v+v^2+o(1).
\]
Substituting the optimizer therefore gives
\begin{equation}
  I_{(\sqrt\beta+\delta)^2,\beta}(\sqrt\beta\,u)
  \longrightarrow\frac{(2\delta-u)^2}{4},
  \label{eq:large-beta-rate-profile}
\end{equation}
uniformly on those compact sets.

We have \(a/\sqrt\beta\to2\delta\).
If \(\delta<\sqrt\gamma\), then \(a<2\sigma\) for all
sufficiently large \(\beta\), and the bulk failure assertion applies.
For \(\delta\geq\sqrt\gamma\), \(a>2\sigma\) for every
\(\beta>0\), and
\begin{equation}
  \frac{\tau}{\sqrt\beta}
  \longrightarrow\delta-\sqrt{\delta^2-\gamma}.
  \label{eq:large-beta-tau}
\end{equation}
This limit lies strictly below \(2\delta\), so
\eqref{eq:large-beta-rate-profile} gives
\begin{equation}
  I_{\alpha,\beta}(\tau)
  \longrightarrow
  \frac{(\delta+\sqrt{\delta^2-\gamma})^2}{4}.
  \label{eq:large-beta-rate}
\end{equation}
The limiting rate is strictly increasing in
\(\delta\geq\sqrt\gamma\), starting at \(\gamma/4\).
It equals one when
\(\delta+\sqrt{\delta^2-\gamma}=2\).
This equation requires \(\delta\leq2\). Squaring and solving gives
\(\delta=1+\gamma/4\), which belongs to this branch precisely
when \(\gamma\leq4\).
For \(\gamma\geq4\), the bulk boundary is \(\delta=\sqrt\gamma\).
This proves the formula \eqref{eq:limiting-sdp-boundary}.
For every strict comparison \(\delta\ne f(\gamma)\), the same
success or failure condition holds for all sufficiently large fixed
\(\beta\). Theorem~\ref{thm:main} then gives the stated conclusion
as \(m\to\infty\).
\end{proof}

\section{Fixed-parameter geometry and figure evaluation}
\label{app:geometry}
\label{sec:geometry}

We derive the boundaries and expansions used in the numerical
illustrations. Write \(\delta=\sqrt\alpha-\sqrt\beta\) and let
\(\delta_{\mathrm{SDP}}(\beta,\gamma)\) denote the boundary in this
coordinate. The limits below describe this boundary after taking the
graph limit at fixed parameters. The model parameters do not vary with \(m\).

\subsection{Monotonicity and the local--bulk junction}

For \(x<a=\alpha-\beta\), write
\[
  \theta_x=\log\frac{x+\sqrt{x^2+4\alpha\beta}}{2\alpha}<0.
\]
Differentiating the optimized expression for the rate gives
\[
  \partial_\alpha I_{\alpha,\beta}(x)=1-e^{\theta_x}>0,
  \qquad \partial_x I_{\alpha,\beta}(x)=\theta_x<0.
\]
At fixed \(\beta,\gamma>0\), \(\tau\) decreases with \(a\) for
\(a>2\sqrt{\beta\gamma}\). Thus
\(I_{\alpha,\beta}(\tau)\) is strictly increasing with \(\alpha\)
throughout this region. It tends to infinity as \(\alpha\to\infty\):
\(\tau=O(\alpha^{-1})\), while
\(I_{\alpha,\beta}(0)=(\sqrt\alpha-\sqrt\beta)^2\), and the
change from zero to \(\tau\) is \(O(\alpha^{-1}\log\alpha)\).
Consequently, the tightness boundary is the bulk value
\(a=2\sqrt{\beta\gamma}\) when the rate there is at least one;
otherwise it is the unique larger \(a\) for which the local rate equals
one. Thus, for fixed \(\beta,\gamma\), there is a unique tightness
boundary as \(\delta\) increases.

For a fixed \(\beta\), consider
\(h_\beta(a)=I_{\beta+a,\beta}(a/2)\), \(a\geq0\).
If \(u=e^{\theta_{a/2}}\), the optimizing equation reads
\(a/2=(\beta+a)u-\beta/u\). For \(a>0\), this implies
\(1/2<u<1\). Differentiating gives
\[
  h_\beta'(a)=1-u+\tfrac12\log u>0.
\]
The last inequality follows because the right-hand side vanishes at
\(u=1\) and has negative derivative on \((1/2,1)\).
Also \(h_\beta(0)=0\) and
\(h_\beta(a)/a\to(1-\log2)/2>0\) as \(a\to\infty\).
There is therefore a unique solution of \(h_\beta(a)=1\).
The local and bulk branches meet at
\begin{equation}
 I_{\beta+a,\beta}(a/2)=1,\qquad \gamma=a^2/(4\beta).
 \label{eq:fixed-beta-junction}
\end{equation} 

\subsection{The small-ratio expansion}

For each fixed \(\beta>0\), as \(\gamma\downarrow0\), the local
boundary has the expansion
\begin{equation}
 \delta_{\mathrm{SDP}}(\beta,\gamma)
 =1+\frac{\beta\log(1+1/\sqrt\beta)}{2(2\sqrt\beta+1)}\gamma+O_\beta(\gamma^2).
 \label{eq:small-gamma-gap}
\end{equation} 

Fix \(\beta>0\), set \(\alpha=(\sqrt\beta+\delta)^2\), and
define near \((\delta,\gamma)=(1,0)\)
\[
  F(\delta,\gamma)=I_{\alpha,\beta}(\tau)-1.
\]
In the formula for \(\tau\), with \(\sigma^2=\beta\gamma\), the
quantity under the square root remains positive for \(\gamma\) near
zero. Thus \(\tau\) and \(F\) are smooth near
\((\delta,\gamma)=(1,0)\). At \(\gamma=0\),
\(\tau=0\) and \(F(\delta,0)=\delta^2-1\). Hence
\(F(1,0)=0\) and \(\partial_\delta F(1,0)=2\).
Moreover,
\[
  \left.\partial_\gamma\tau\right|_{\gamma=0}
  =\frac{\beta}{\alpha-\beta},\qquad
  \partial_x I_{\alpha,\beta}(0)
  =-\log\left(1+\frac{\delta}{\sqrt\beta}\right).
\]
The implicit function theorem gives a smooth local boundary
\(\delta_{\mathrm{SDP}}(\beta,\gamma)\) with
\[
  \left.\frac{\partial\delta_{\mathrm{SDP}}}{\partial\gamma}
  \right|_{\gamma=0}
  =-\frac{\partial_\gamma F(1,0)}{\partial_\delta F(1,0)}
  =\frac{\beta\log(1+1/\sqrt\beta)}{2(2\sqrt\beta+1)}.
\]
For small positive \(\gamma\), the bulk inequality is strict near
\(\delta=1\), and the monotonicity just proved identifies this
implicit curve with the SDP recovery boundary. Taylor's theorem proves
\eqref{eq:small-gamma-gap}. The coefficient is positive; its limit
\(1/4\) follows from \(\log(1+s)=s+O(s^2)\) as \(s\downarrow0\).

Figure~\ref{fig:background-dependence} compares the boundaries at
three fixed values of \(\beta\). Its right panel plots \(c_\beta\),
the coefficient of \(\gamma\) in \eqref{eq:small-gamma-gap}, and
shows how the initial slope of the gap approaches \(1/4\).

\begin{figure}[t]
  \centering
  \includegraphics[width=\linewidth]{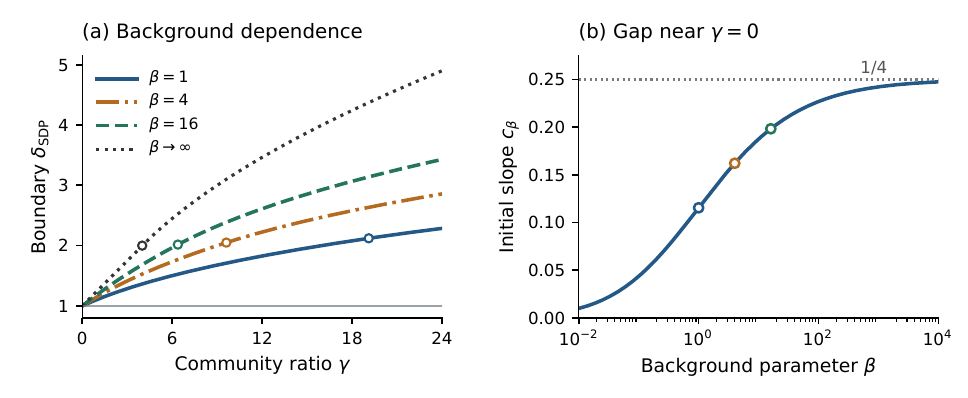}
  \caption{Dependence of the tightness boundary on the background parameter.
  (a) Boundaries for \(\beta=1,4,16\), together with the sequential
  limit \(f(\gamma)\) from \eqref{eq:limiting-sdp-boundary}.
  Circles mark the local--bulk junctions; the horizontal line is the
  information boundary \(\delta=1\).
  (b) The initial slope \(c_\beta\) from \eqref{eq:small-gamma-gap}.
  Colored circles identify the same three values of \(\beta\) as in
  panel~(a), and the dotted line marks the limit \(1/4\).
  Both panels evaluate asymptotic formulas after the graph limit at
  fixed parameters. Recovery at the critical boundaries is not asserted.}
  \label{fig:background-dependence}
\end{figure}

\subsection{Evaluation of the theoretical curves}

The numerical figures evaluate formulas stated in the paper. The source
\texttt{README.md} lists the plotting commands, and the accompanying
scripts and curve tables specify the parameters and grids. Numerical
checks compare boundary roots with their defining equations, the rate
function with its cumulant expression, forest corrections with
tridiagonal inverses, spectral densities with their exact mass and mean,
and refinement recurrences with their geometric upper bound.
The boundary at \(\gamma=0\) is plotted by continuity. These evaluations
illustrate the analytic results; they use no sampled graphs or numerical
SDP solutions.

%% file: appendices/comparison.tex
\section{Comparison of models and feasible sets}
\label{app:comparison}

To compare recovery guarantees, we must account for differences in
the parameter scaling and the feasible sets. For example, if a paper uses
\(p=a_{\rm old}\log n/n\), its parameter in our scaling is
\(a_{\rm old}=\alpha k\log m/\log n\), which grows with \(m\).
Applying a theorem stated for fixed \(a_{\rm old}\) therefore requires
checking that its bounds remain valid for these growing parameters.

\paragraph{The same balanced relaxation.}
Theorem~1.1 and SDP (4) of \citet{AgarwalBandeiraKoiliarisKolla15}
use the same \((m,k,\alpha,\beta)\) scaling and feasible set as
\eqref{eq:sdp}. Their theorem allows \(k=O(\log m)\): it is sharp
for \(k=o(\log n)\), whereas its sufficient condition at
\(k=(\gamma+o(1))\log n\) includes a universal constant.
Their Section~3 and Appendix~A.6 give the certificate displayed in
Appendix~\ref{sec:success}. Proposition~\ref{prop:marked-resolvent}
sharpens the sufficient condition obtained from this certificate.

Section~6 of \citet{HajekWuXu16} uses the signed matrix formulation
with \(Y_{ii}=1\), \(Y\succeq0\), \(Y\one=0\), and
\(Y_{ij}\geq-1/(k-1)\). The transformation
\[
  Y=\frac{kZ-J}{k-1},\qquad Z=\frac{(k-1)Y+J}{k}
\]
maps its feasible set bijectively to \eqref{eq:sdp}. On
\(\one^\perp\), it rescales positive semidefiniteness by a
positive factor; the row constraints determine the action on \(\one\).
The objective also changes by a positive factor and an additive constant.
Their Theorem~17 and Remark~18 establish a gap for this same SDP when
the logarithmic ratio is sufficiently large. They do not determine
the boundary for every fixed positive ratio.
Their Section~7 also discusses local and global perturbations.
The present analysis determines the boundary produced by these
mechanisms at each fixed positive ratio.

\paragraph{A larger trace-norm relaxation.}
In \citet[Algorithm~2, Theorems~6 and~8]{ChenXu16}, specialize the
cluster count and size to \(r=k\), \(K=m\), with no unclustered
vertices. Their convex program requires
\(\|Z\|_*\leq n\), \(\langle J,Z\rangle=nm\), and
\(0\leq Z_{ij}\leq1\), where \(\|\cdot\|_*\) is the nuclear
norm. Every feasible point of \eqref{eq:sdp} satisfies these constraints:
positive semidefiniteness and unit diagonal give
\(\|Z\|_*=\Tr Z=n\) and \(|Z_{ij}|\leq1\), and the row sums
give total mass \(nm\). Their feasible set is therefore larger.
If their relaxation uniquely recovers the planted matrix, then ours
does as well. Nonoptimality does not transfer in the same direction:
a better matrix in their feasible set need not belong to ours. Remark~7
and Section~2.3.1 of their paper explicitly distinguish these directions.
Their sufficient condition has the leading form
\(a^2\geq c(\alpha+\beta\gamma)\), with an unspecified universal
constant~\(c\).

\paragraph{Other mechanisms and recovery objectives.}
Theorems~4--5 of \citet{HajekWuXuExtensions16} and Section~1.2.3
of \citet{PerryWein16} keep the community count fixed; the latter
allows unequal fixed proportions. Theorem~1 of \citet{GaudioProngpaophan25}
shows nonoptimality of a binary symmetric SDP for some unequal
within-community edge probabilities; their Section~2 explicitly
distinguishes this from failure after rounding. Our model keeps the
within-community probabilities equal and varies the community count.

\citet{MontanariSen16} and \citet{FanMontanari17} study
SDP values and detection at bounded average degree. Our average degree
diverges as \((\log m)^2\), and our feasible set additionally imposes
the entrywise and row-sum constraints in \eqref{eq:sdp}.
Our Gram constructions explicitly enforce these additional constraints.

\paragraph{Error exponents and refinement.}
Section~2.2 of \citet{FeiChen19} explicitly permits replacing a total-mass
constraint by the row constraints in \eqref{eq:sdp}; their upper-entry
bound is already implied by positive semidefiniteness and unit diagonal.
Their results therefore apply to the same relaxation.
Theorems~1--2 give exponential matrix-error bounds and a rounding
guarantee using approximate \(k\)-medians. The constants in those
exponents are not optimized; our Theorem~\ref{thm:error} determines the
matching exponent \(I(\tau)\) in the present scaling.
\citet[Section~3.2 and Theorem~2]{FeiChenBayes19} give a sharp
Bayes exponent for the binary symmetric SBM. Their Section~1.1
also uses a planted dual certificate to bound errors without assuming
tightness. Our diagonal-adjustment argument builds on this approach.

\citet{GaoMaZhangZhou17} separate initialization and local refinement.
Their Condition~1 requires a suitable error bound and failure
probability for each leave-one-out initialization; Theorem~4 then
controls the refined error. Their concrete spectral initialization
in Theorems~6 and~13 does not directly cover our scaling:
with \(a_{\rm old}=np=\alpha kL\) and \(b_{\rm old}=nq=\beta kL\),
\[
  \frac{(a_{\rm old}-b_{\rm old})^2}{a_{\rm old}k^3}
  =\frac{(\alpha-\beta)^2L}{\alpha k^2}\longrightarrow0.
\]
Here a coarse SDP error bound gives a sufficiently accurate initial
labeling. We analyze refinement on an event that holds for all small
error sets. This allows the same graph to be used at every step despite the dependence
of the initial labels on its edges.

Nonintegral SDP solutions can also reveal clusters in other settings.
\citet{ZurekChen24} recover large clusters in the presence of
smaller and critically sized clusters and analyze noninteger
solution blocks. Their heterogeneous-size setting differs from
the fixed-balanced error law studied here.

\paragraph{Growing-community KS results.}
The v3 formulation of \citet{ChinMosselSohnWein26} assigns labels
independently and uniformly. Its normalized signal-to-noise ratio has the formal limit
\(a^2/(\beta\gamma)\) under our edge scaling.
However, its main algorithmic regime assumes
\(d,\lambda=\Theta(1)\), whereas here
\[
  d=(\alpha+(k-1)\beta)L\asymp L^2,\qquad
  \lambda=\frac{\alpha-\beta}{\alpha+(k-1)\beta}\asymp L^{-1}.
\]
Its Theorem~1.3 allows broader densities for low-degree hardness,
but remains an independent-label recovery result.
Transferring these results to fixed balanced partitions would require
control under conditioning on exact balance, as well as a change from
weak to exact recovery. This signal-to-noise comparison therefore does
not identify \(a=2\sigma\) as a computational threshold for our model.